\documentclass[11pt]{article}
\usepackage{color}
\usepackage{amssymb,amsfonts,amsthm, mathtools, enumitem}
\usepackage{todonotes}
\usepackage[english]{babel}
\usepackage[utf8]{inputenc}
\usepackage[margin=24mm]{geometry}
\usepackage{graphicx,overpic}
\usepackage{times,fourier}
\usepackage{tikz}
\usetikzlibrary{shapes, calc, math, fpu}
\usepackage{pgf}
\usepackage{pgfplots}
\pgfplotsset{compat=newest}
\usetikzlibrary{decorations.markings}
\tikzset{->-/.style={decoration={ markings, mark=at position #1 with {\arrow{>}}},postaction={decorate}}}
\usepackage{float}
\usepackage[utf8]{inputenc}
\usepackage[absolute]{textpos}
\usepackage{ulem}
\usepackage{emptypage}
\usepackage{amsbsy,mathrsfs}
\usepackage{subcaption}
\usepackage{xcolor}
\usepackage{multicol}
\usepackage{thmtools, thm-restate}
\usepackage[colorlinks=true,citecolor=blue, hypertexnames=false]{hyperref}

\def\ov{\overline}
\makeatletter
\renewcommand\section{\@startsection {section}{1}{\z@}%
	{-2ex \@plus -1ex \@minus -.2ex}%
	{1ex \@plus.1ex}%
	{\normalfont\bf\sffamily\color{darkblue}}}
\renewcommand\subsection{\@startsection{subsection}{2}{\z@}%
	{-1.75ex\@plus -0.4ex \@minus -.2ex}%
	{0.6ex \@plus .1ex}%
	{\normalfont\small\bf\sffamily}}
\renewcommand\subsubsection{\@startsection{subsubsection}{3}{\z@}%
	{-0.6ex\@plus -0.2ex \@minus -.2ex}%
	{0.4ex \@plus .1ex}%
	{\normalfont\normalsize\it}}
\renewcommand\paragraph{\@startsection{paragraph}{4}{\z@}%
	{0.2ex \@plus0.2ex \@minus0.1ex}{-0.5em}%
	{\normalfont\normalsize\bfseries}}
\def\ps@headings{%
	\let\@oddfoot\@empty
	\let\@evenfoot\@empty

	\def\@evenhead{\small\sffamily\thepage\hfil\slshape\leftmark}%
	\def\@oddhead{\small\sffamily{\slshape\rightmark}\hfil\thepage}%
	\let\@mkboth\markboth
	\def\chaptermark##1{\markboth{{\ifnum \c@secnumdepth >\m@ne
				\if@mainmatter \@chapapp\ \thechapter. \ \fi \fi ##1}}{}}%
	\def\sectionmark##1{\markright {{\ifnum \c@secnumdepth >\z@
				\thesection. \ \fi ##1}}}}
\def\fbf#1{\setbox0=\hbox{$#1$}\kern-0.10\wd0
	\lower0.02em\copy0\kern-\wd0 \lower0.02em\hbox{\kern+0.04em\copy0}\kern-\wd0
	\raise0.00em\copy0\kern-\wd0 \raise0.00em\hbox{\kern-0.04em\box0}}
\def\overl@ss#1#2{\vcenter{\offinterlineskip
		\ialign{$\m@th#1\hfil##\hfil$\crcr#2\crcr<\crcr } }}
\def\gl{\mathrel{\mathpalette\overl@ss>}}
\makeatother

\advance\textwidth -10mm
\advance\textheight -14mm
\advance\hoffset 4mm
\advance\voffset 4mm

\advance\parskip 4pt
\numberwithin{equation}{section}
1

\declaretheorem[name=Theorem, parent=section]{theorem}
\declaretheorem[name=Lemma,sibling=theorem]{lemma}

\declaretheorem[name=Remark, sibling=theorem]{remark}

\declaretheorem[name=Proposition, sibling=theorem]{proposition}
\declaretheorem[name=Assumption, sibling=theorem]{assumption}

\declaretheorem[name=RHP, parent=section]{RHP}

\newcommand{\restr}[2]{#1\raisebox{-.5ex}{$\bigg|$}_{#2}}

\makeatletter
\newcommand{\address}[1]{\gdef\@address{#1}}
\gdef\@address{} 

\renewcommand{\maketitle}{%
	\begin{center}
		{\LARGE\bfseries\sffamily \@title\par}
		\vspace{1.4ex}
		{\large \@author\par}
		\vspace{0.6ex}
		{\itshape \@address\par}
		\vspace{0.2ex}
		{\small \@date\par}
	\end{center}
	\vspace{1.4\bigskipamount}
}
\makeatother

\def\be{\begin{equation}}
	\def\ee{\end{equation}}
\def\bse{\begin{subequations}}
	\def\ese{\end{subequations}}

\definecolor{deeppurple}{rgb}{0.5, 0, 0.7}
\definecolor{deeppurple}{rgb}{0.5, 0, 0.7}
\definecolor{darkblue}{rgb}{0, 0, 0.7}

\definecolor{deeppurple}{rgb}{0.5, 0, 0.7}

\def\half{{\textstyle\frac12}}

\def\dn{\mathop{\rm dn}\nolimits}

\def\Real{\mathbb{R}}
\def\R{\mathbb{R}}
\def\Complex{\mathbb{C}}

\def\E{\mathcal{E}}
\newcommand{\bigo}[1]{\mathcal{O} \left( #1 \right) }

\def\i{\text{i}}

\def\Re{\mathop{\rm Re}\nolimits}
\def\Im{\mathop{\rm Im}\nolimits}
\def\Res{\mathop{\rm Res}\limits}

\def\d{\mathrm{d}}

\let\t=\theta
\def\e{\mathop{\rm e}\nolimits}
\def\sn{\mathrm{sn}}
\def\cn{\mathrm{cn}}
\def\dn{\mathrm{dn}}
\def\@#1{{\mathbf{#1}}}
\def\_#1{{\mathsf{#1}}}
\let\~=\tilde
\let\==\bar
\let\^=\hat
\let\<=\langle
\let\>=\rangle

\def\min{\mathop{\rm min}\nolimits}

\def\note[#1]{\marginpar{\color{blue}[#1]}}

\def\C{{\mathbb C}}
\def\R{{\mathbb R}}
\def\N{{\mathbb N}}

\def\1{{\bf 1}}
\def\s{\sigma}

\def\l{\lambda}

\def\e{\mathrm{e}}

\makeatother

\advance\textheight 12mm
\advance\voffset -6mm
\allowdisplaybreaks

\def\br{\begingroup\color{red}}
\def\er{\endgroup}

\newcommand{\zeroset}{\mathrm{Z}}

\renewcommand{\L}{\mathrm{L}}
\newcommand{\Msol}{M^{\mathrm{sol}}}

\begin{document}
\pagestyle{plain}

\title{Inverse scattering for the focusing  nonlinear Schr\"odinger equation with elliptic  background and full soliton gas  }
\author{Tamara Grava $^1$, Robert Jenkins $^2$, Xiaofan Zhang $^{3}$ and Zechuan Zhang $^1$}
\address{$^1$ Scuola Internazionale Superiore di Studi Avanzati (SISSA), Trieste, 34136 - Italy\\ 
	$^2$ 
	University of Central Florida, Orlando, FL 32816 - USA\\
	$^3$ China University of Mining and Technology, Xuzhou 221116 - China}
\maketitle

\begin{abstract}
\noindent
In this manuscript 
we develop the direct and inverse scattering problem for the  cubic focusing nonlinear Schr\"odinger equation
and for  initial data  that are asymptotic to an  elliptic  travelling wave with  distinct phase at $\pm \infty$.
We consider the case in which the spectral bands intersect the real axis.
We then show that  this class of initial data  has non zero intersection with the  full soliton gas initial data.

\end{abstract}

\medskip
\tableofcontents

\section{Introduction and outline}
\label{s:intro}

In this manuscript we study the direct and inverse scattering transforms associated to the Cauchy problem for the cubic focusing nonlinear Schr\"odinger (NLS) equation
\be
\i u_t + \half u_{xx} +  |u|^2 u = 0\,,\quad u(x,0) = u_0(x),  \qquad  (x,t) \in \R^2. 
\label{e:nls}
\ee 
for initial data $u_0$ which approaches a step-like elliptic background (cf. \eqref{initial_data}).
Here $u = u(x,t)\in\C$ and subscripts denote partial derivatives. The term `background' here refers to a function $u_b(x,0)$ which our initial data approaches asymptotically. 
Typically one knowns the explicit evolution $u_b(x,t)$ of the background under evolution by \eqref{e:nls}. 
The Cauchy problem on the zero background \cite{BC, DZ03} and for nonzero constant backgrounds \cite{BG14, DPVV14, FT} are well-studied.
Much less is known about the behavior of solutions which approach non-constant background states.
The motivation to study such problem stems from the fact that the development of direct and inverse scattering for such initial data enables to study long time  behaviour of elliptic travelling waves under a broad range of perturbations.  
While the family of linearly and orbitally stable perturbations of elliptic travelling wave are well known, (see, e.g. \cite{GH1}\cite{Deconinck}), it is not known what happens to families of unstable  perturbation of  elliptic  travelling waves  over long time. This is the main motivation of our present study. 
 
The cubic NLS equation \eqref{e:nls} admits a family of quasi-periodic solutions known as finite-gap solutions. 
The simplest of these solutions (the genus one case) can be expressed  by elliptic functions.  The explicit form  is given by
\begin{gather}
\label{e:elliptic0}
	u_e(x, t; x_0, \varphi_0) := \br\lambda\er U(\lambda (x-x_0 -v t);\lambda^2 t) \e^{\i v \left((x-x_0) - \frac{v}{2} t\right) +2\i \varphi_0}\\
	\label{e:ellipticU}
	U(x;t)=\sqrt{ \left[ b- m\, \sn^2(x ; m)    \right] } \e^{\i \phi(x)-\i  \widetilde{\omega_0} t}\\
	\label{e:phi}
	\phi(x)=\int_{0}^x\dfrac{ \sqrt{b(1-b)(b-m)}\d y}{\left[ b- m\, \sn^2( y  ; m)    \right]  }\,, \qquad \widetilde{\omega_0} = \frac{1+m-3b}{2}
\end{gather}
where $\sn(\,\cdot\,;m)$  is the Jacobi elliptic function of modulus $m$. The six parameters $(m,b,v, \lambda,x_0,\varphi_0)$ satisfy $0<m<b<1$, $\lambda>0$, and $v, x_0, \varphi_0 \in \R$. 
The symmetries of \eqref{e:nls}:
\begin{enumerate}
  \item $u(x,t) \mapsto u(x,t)\,\e^{\i \varphi_0}$, $\varphi_0\in\R$
  (Phase invariance);
  \item $u(x,t) \mapsto u(x-x_0,t)$, $x_0\in\R$ (Translation
    invariance);
  \item $u(x,t)\mapsto  \e^{\i \big( vx - \frac{v^2}2t\big)}
    u(x-vt,t)$, $v\in\R$ (Galilean invariance);
  \item $u(x,t) \mapsto \lambda u(\lambda x,\lambda^2t)$, $\lambda > 0$
    (Dilation invariance),
\end{enumerate}
allow one to reduce the six parameter family of solutions \eqref{e:elliptic0} to \eqref{e:ellipticU} depending only on the two parameters $(m,b)$.  In what follows,  in order to have formulas which directly apply to the general case we will work with general solutions \eqref{e:elliptic0}

The cubic NLS  equation is integrable, namely it can be expressed as the compatibility condition of two linear equations, the so-called Zakharov-Shabat (ZS) pair,  introduced in  1972 \cite{ZS72}.
This pair takes the form 
\bse%
\label{e:NLSLP}
\begin{align}
	&W_x = \mathcal{L}(u,z)\,W\,, &&\mathcal{L}(u,z) = -\i z\sigma_3 + U(x,t)\,,
	\label{e:zs}
	\\
	&W_t = \mathcal{B}(u,z)\,W\,,  &&\mathcal{B}(u,z) = -\i z^2\sigma_3 + zU - \frac{1}{2}\i\sigma_3(U^2-U_x)\,,
	\label{e:NLSLP2}
\end{align}
\ese
where $W(x,t;z)\in  \mbox{Mat}(2\times 2,\Complex)$ is the fundamental matrix solution  of the above linear system,  $z\in\Complex$ is the spectral parameter,  and  
\be
U(x,t)=\begin{pmatrix}
	0 & u(x,t)\\
	-\overline{u(x,t)} & 0
\end{pmatrix}, \qquad \sigma_3=\begin{pmatrix*}[r]1&0\\0&-1\end{pmatrix*}.
\ee

In the spectral plane, the finite-gap solution \eqref{e:elliptic0} is parameterized by two distinct points $z_1, z_2 \in \C^+$ and the phase constants $x_0, \varphi_0 \in \R$. 
There is a one-to-one map between the complex parameters $z_1, z_2$ and the four real parameters $m,b,\lambda,v$:
\be\label{e:modulus}
	m=1-\left|\frac{z_1-z_2}{z_1-\overline{z_2}}\right|^2,  \qquad 
	b=\left( \dfrac{\Im(z_1-\overline{z_2})}{|z_1-\overline{z_2}|}\right)^2, \qquad 
	\lambda = |z_1 -\overline{z_2}|,\qquad 
	v=-\Re(z_1+z_2)\,.
\ee 
The ZS scattering problem \eqref{e:zs},  with potential of the form  \eqref{e:elliptic0},   generates a continuous spectrum  on the real line and two bands, which form the set $\Sigma$ consisting of  two finite arcs connecting $z_1, z_2, \bar z_1, \bar z_2$. 
The precise topology of these arcs depends on the choice of $z_1$ and $z_2$ (see Proposition~\ref{prop:sigma}).


In this work, we study the Cauchy problem for \eqref{e:nls} given initial data that approaches different elliptic travelling waves at either spatial infinity. 
These waves have the same spectral parameters $z_1, z_2$, but different phase parameters $x_0$ and $\varphi_0$.
For this reason we indicate the elliptic travelling wave as $u_e(x,t,x_0,\varphi_0)$. We assume that 
 \be
 \label{initial_data}
u(x,0) = u_0(x)\to\left\{
\begin{array}{ll}
 u_e^\ell(x)=u_e(x,0,x^\ell_0,\varphi^\ell_0),&\qquad \mbox{as $x\to-\infty$}\\
 u_e^r(x)= u_e(x,0,x^r_0,\varphi^r_0),&\qquad \mbox{as $x\to\infty$,}
 \end{array}
 \right.
\ee
where $x_0^\ell, x_0^r$ and $\varphi_0^\ell,\varphi_0^r$ are in general, distinct phase shifts. Both limiting elliptic waves generate a continuous spectrum of the associated ZS operator, denoted by $\Sigma^\ell\cup\R$ and $\Sigma^r\cup\R$, respectively.
In our previous work \cite{GJZZ} we chose initial data such that: 1) the arcs $\Sigma^\ell$ and $\Sigma^r$ were independent with distinct endpoints; 2) neither spectral bands was permitted to intersect the real axis.
In this manuscript we choose initial data such that: 
\begin{itemize} 
\item  $\Sigma^\ell\equiv \Sigma^r=\Sigma$;
\item $\Sigma$  is permitted to intersects the real line  (see Figure~\ref{fig:cuts}).
\end{itemize}
For illustration, Figure~\ref{fig:numerics} shows two numerical examples of step-like elliptic backgrounds with coincident spectral bands.
This does not exhaust the zoo of all possible configurations of the spectrum, but the remaining configurations are easily adapted from this work and our previous manuscript \cite{GJZZ}.  

\begin{figure}[t]
	\centering
	\begin{overpic}[width=0.45\linewidth]{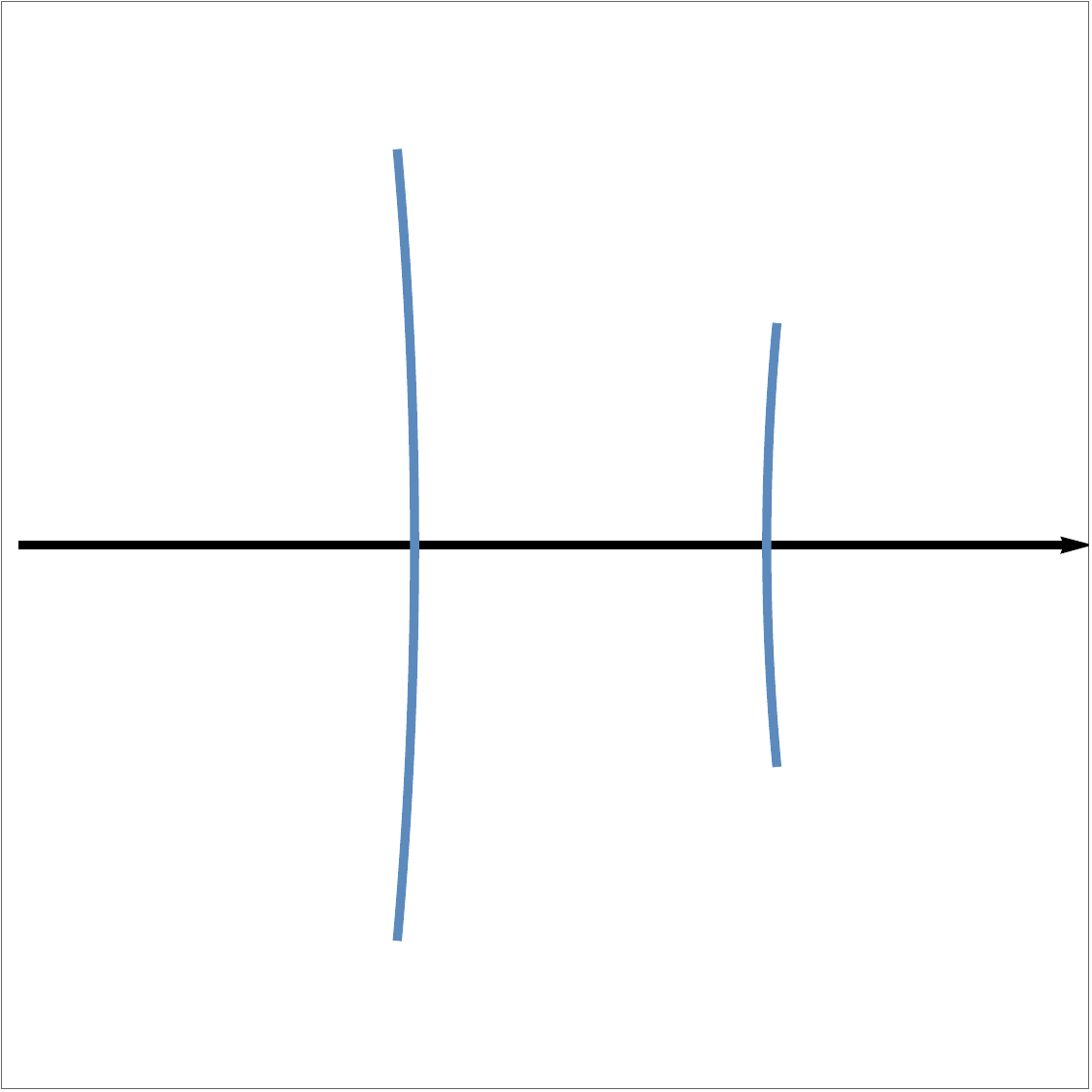}
		\put(34,85){\makebox(0,0)[r]{$z_1$}}
		\put(34,14){\makebox(0,0)[r]{$\overline z_1$}}
		
		\put(74,70){\makebox(0,0)[l]{$z_2$}}
		\put(74,28){\makebox(0,0)[l]{$\overline z_2$}}
		
		\put(35,85){\color{red}$\bullet$} 
		\put(35,12.5){\color{red}$\bullet$}  
		\put(70,69){\color{red}$\bullet$}  
		\put(70,28){\color{red}$\bullet$}  
		\put(36.75, 48.75){\color{red}$\bullet$}  
		\put(69., 48.75){\color{red}$\bullet$}  
		\put(94,53){\color{black}$\mathbb{R}$}
		
		\put(29,63){$\Sigma_1^+$}
		\put(29,30){$\Sigma_1^-$}
		\put(72,57){$\Sigma_2^+$}
		\put(72,40){$\Sigma_2^-$}
		\put(39, 45){\small{$\xi_1$}}
		\put(65.25,45){\small{$\xi_2$}}
	\end{overpic}
	\caption{
	The spectral bands $\Sigma=\Sigma_1 \cup \Sigma_2$ which define the spectrum of the finite-gap solutions (see \eqref{e:elliptic0} and \eqref{initial_data}) which define the background of our initial data. 
	We use $\pm$ superscripts to denote the part of each band in the upper/lower complex half-planes: $\Sigma_k^\pm = \Sigma_k \cap \mathbb{C}^\pm$, $k=1,2$. The points $\xi_k,\ k=1,2$, denote the intersection points of each $\Sigma_k$ with $\R$. 
	}
	\label{fig:cuts}
\end{figure}

\begin{figure}[!htbp]

\begin{subfigure}[t]{0.45\textwidth}
	\centering
	\begin{overpic}[
		width=7.3cm,
		height=5.5cm
		]{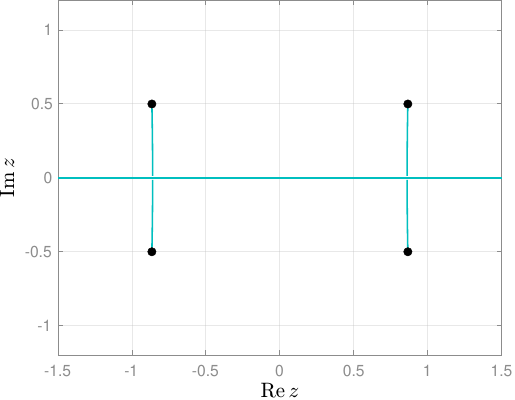}
		\put(32,57){$z_1$}
		\put(82,57){$z_2$}
		\put(32,27){$\bar{z}_1$}
		\put(82,27){$\bar{z}_2$}
	\end{overpic}
	\caption{}
\end{subfigure}
\hfill
\begin{subfigure}[t]{0.55\textwidth}
	\centering
	\includegraphics[
	width=7.3cm,
	height=5.5cm
	]{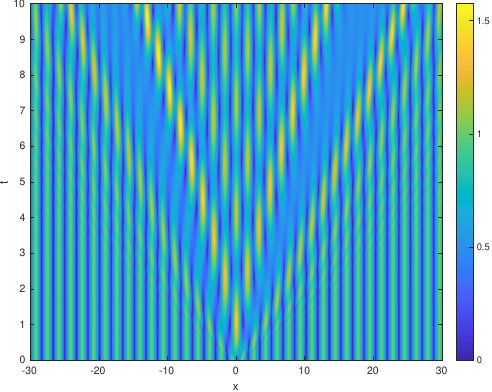}
	\caption{}
\end{subfigure}

\par\vspace{0.2em}

\begin{subfigure}[t]{0.45\textwidth}
	\centering
	\includegraphics[
	width=7.3cm,
	height=5.5cm
	]{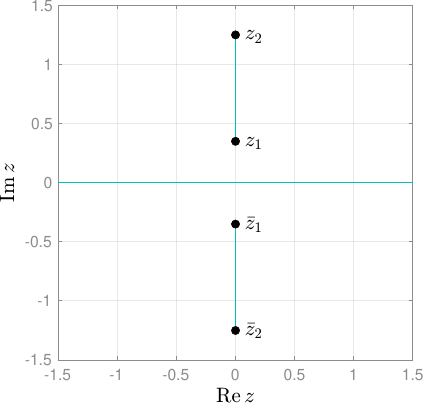}
	\caption{}
\end{subfigure}
\hfill
\begin{subfigure}[t]{0.55\textwidth}
	\centering
	\includegraphics[
	width=7.3cm,
	height=5.5cm
	]{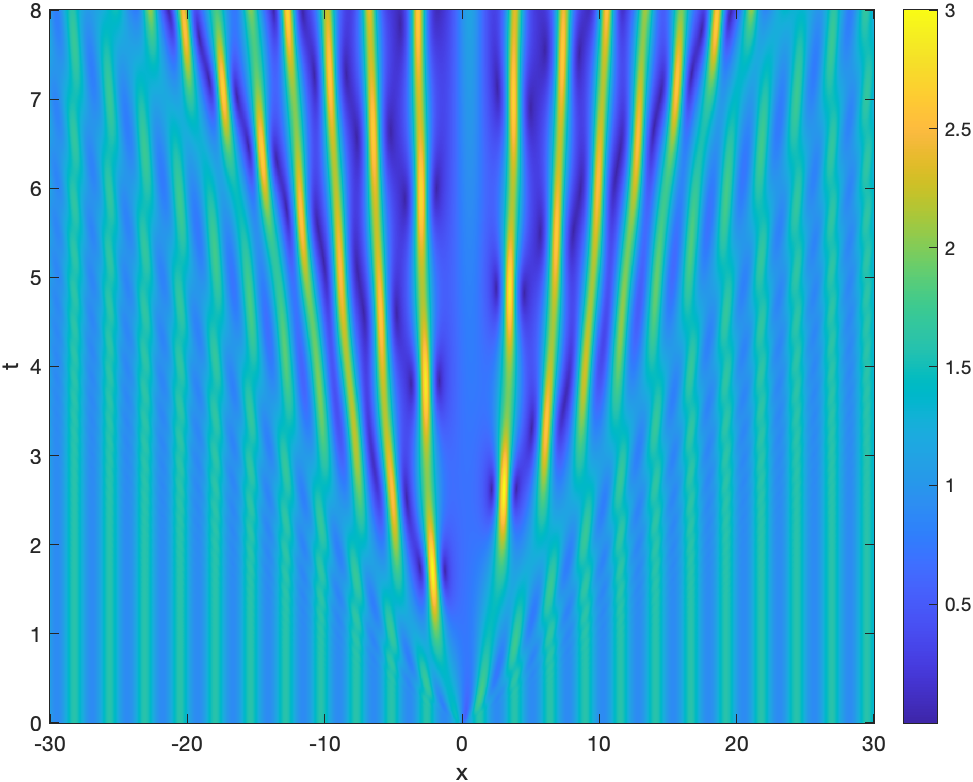}
	\caption{}
\end{subfigure}
	\caption{Numerical simulations for step-like initial data connecting elliptic
		backgrounds with coincident spectral bands. In both cases, the left and
		right backgrounds have the same elliptic parameters, and the perturbation is
		introduced only through different spatial and phase shifts $(x_0,\phi_0)$.
		The first row corresponds to the $\cn$ case, for which we use $(x_0^L,\phi_0^L)=(2,0.3)$ and
		$(x_0^R,\phi_0^R)=(0,0)$. The second row corresponds to the $\dn$ case, for which we use
		$(x_0^L,\phi_0^L)=(1.28,0.85)$ and
		$(x_0^R,\phi_0^R)=(0,0)$. For each row, the left figure shows the corresponding spectral bands, and the right figure shows the space--time density plot of $|q(x,t)|$.
	}
	
	\label{fig:numerics}
\end{figure}

\paragraph{Statement of the result.}
To state our results we introduce standard Sobolev spaces on the real halt-lines $\R^\pm$,  
\[
	\mathcal{W}^{n,1}(\R^\pm) = \left\{ f \in L^1(\R^\pm) \, \mid \, \partial_x^j f \in L^1(\R^\pm), 0 \leq j \leq n  \right\},
\]
and the weighted spaces 
\[
L^{p,k}(\R^\pm)=\left\{f\in L^{p}(\R^\pm),\;\; |x|^k f \in L^{p}(\R^\pm)\right\},\quad p,k\in\N.
\]
%
\begin{assumption}
\label{ass1}  The initial data $u_0$ satisfies \eqref{initial_data} in the strong sense that  $u_0 - u_e^\ell \in \mathcal{W}^{4,1}(\R^-)$
and $u_0 - u_e^r \in \mathcal{W}^{4,1}(\R^+)$. 
\end{assumption}
\begin{assumption}
\label{ass2}
The initial data $u_0$ is generic in the sense of  \cite{BC}, i.e., it produces a discrete spectrum that is simple, finite, and contains no points in $\R$. 
\end{assumption}

We consider initial data satisfying assumptions \ref{ass1} and \ref{ass2}, with  the spectral bands $\Sigma$  that intersect the real line.
%
%
 The main results of this manuscript are:
 \begin{itemize}
 \item the   calculation of the direct spectral map between initial data and scattering data;
 \item the  realization of the scattering data for the inverse problem as a limit of infinite number of solitons that we call the full gas problem.
 \end{itemize}
The soliton gas  solutions were introduced in \cite{GGJM}, \cite{GGJMM} for the Korteweg-de Vries  and modified Korteweg de Vries equations  and 
in \cite{BGO1}, \cite{BGO2}, \cite{BGO3} for the nonlinear Schr\"odinger equation. In all these cases solitons are so dense that are no more distinguishable from one another and   they form a condensate that has non-zero density for  $x\to+\infty$ or  $x\to-\infty$.  The so called {\it full gas solution}, namely the case in which solitons have non zero density as $x\to+\infty$ and  $x\to-\infty$  was derived in \cite{GJM} for the modified Korteweg-de Vries  equation and it is derived in the present manuscript for the NLS  equation. In particular our result  shows that   any  step-like initial data  of the form \eqref{initial_data}
 can be  obtained from a full  soliton gas. Furthermore the primitive potential   formally introduced by Dyachenko, Zakharov and Zakharov,  \cite{DZZ}  arises naturally from step-like potential and can be seen as a full soliton gas.

 \begin{theorem}[Direct problem]\label{thm:direct}
 	Given initial data $u_0$ of the type \eqref{initial_data} that satisfies Assumptions~\ref{ass1} and \ref{ass2}, the scattering map takes the form
	$u_0 \mapsto  \left(  r_1, r_2, \rho, \{ (z_j, c_j) \}_{k=1}^N  \right)$ where
	\begin{itemize}
\item  $r_1$ and $r_2$ are continuous functions supported on $\Sigma^+ = \Sigma \cap \C^+$, i.e.,  $r_1, r_2\in C(\Sigma^+)$;
\item  $\rho\in L^{2,2}(\R)\cap L^{1,2}(\R)$;
\item  The finite set of discrete points $z_j\in\C^+\setminus \Sigma$ are $L^2$ eigenvalues of the ZS scattering operator \eqref{e:zs} with associated norming constant $c_j\in\C\backslash\{0\}$, $j=1,\dots, N$.
\end{itemize}
	
 \end{theorem}
Given the spectral data, we express the inverse problem at any time $t\in \R$ as a Riemann-Hilbert problem that encodes the time-evolved solution of the NLS equation.
By Assumption~\ref{ass2} the discrete spectrum is finite and does not intersect $\Real\cup\Sigma$. To set the Riemann-Hilbert problem we orient $\R$ left-to-right. 
When the spectral bands $\Sigma$ cross the real axis (Figure~\ref{fig:cuts}) we orient each connected component upward; when $\Sigma$ does not cross the real axis (Figure~\ref{fig:cuts1}) we orient $\Sigma^+$ from $z_1$ to $z_2$ and $\Sigma^-$ from $\bar z_2$ to $\bar z_1$. In either case $\Sigma$ is anti-symmetric, i.e.,  $\overline{\Sigma} = - \Sigma$. 

 \begin{RHP}\label{RHP:Mtilde}
 Given the data 	described in Theorem \ref{thm:direct},  
	find a $2\times2$ matrix-valued function $M(z;x,t )$ which satisfies the following conditions:
	\begin{enumerate}
		\item $M(z;x,t)$ is analytic for $z \in \Complex\setminus (\Real\cup\Sigma_1\cup\Sigma_2)$  with  simple poles at the points $\{z_j,\overline{z_j}\}_{j=1}^N$;
		\item $M(z;x,t)=I+\mathcal{O}(z^{-1})$, as $z\to\infty$\,;
		\item $M(z;x,t)$ satisfies Schwarz symmetry: $\overline{M(\overline{z};x,t)}= \sigma_2 M(z;x,t) \sigma_2$, 
		   where $\sigma_2= \begin{psmallmatrix*}[r] 0& -\mathrm{i} \, \\ \mathrm{i} & 0 \end{psmallmatrix*}$;
		\item $M(z;x,t )$ satisfies the jump condition $M(z_+;x,t)=M(z_-;x,t) V(z;x,t)$ for $z\in \R\cup  \Sigma$ 
		 where $ V(z;x, t)$ is given by
		\be\label{e:V0}
		V(z;x,t)=\begin{cases}
				\begin{pmatrix}
					\dfrac{1-r_1(z)r_2(z)}{1+r_1(z)r_2(z)} & \dfrac{2\mathrm{i} r_2(z)}{1+r_1(z)r_2(z)}\e^{-2\theta(z;x,t)} \vspace{0.5ex}
					\\[0.8em]
					\dfrac{2\mathrm{i} r_1(z)}{1+r_1(z)r_2(z)}\e^{2\theta(z;x,t)} & 	\dfrac{1-r_1(z)r_2(z)}{1+r_1(z)r_2(z)} 
				\end{pmatrix}, & z\in\Sigma \cap \C^+ ,			\vspace{6pt}
			\\
			\begin{pmatrix}
				1+|\rho(z)|^2 & \rho^*(z)\e^{-2\theta(z;x,t)} \vspace{0.5ex}\\
				\rho(z)\e^{2\theta(z;x,t)} & 1
			\end{pmatrix}, & z\in \Real, 
		\end{cases}
		\ee
		and for $z\in \Sigma\cap \C^-$,  $V(z;x,t) = \sigma_2 \overline{V(\bar z;x,t)} \sigma_2$. 
		\begin{equation}\label{e:theta}
			\theta(z;x,t) = \mathrm{i}t z^2+\mathrm{i}x z\,;
		\end{equation}
		 \item $M(z;x,t)$ has simple poles at the  points  $z_k$,  $\ov{z}_k $, $k=1,\dots, N$,   with residues
		\begin{equation}	\label{e:residue}
			\underset{z = z_k}{\Res}\, M(z;x,t) 
			= \lim_{z \to z_k} M(z;x,t) 
			\begin{pmatrix} 0 & 0 \\ 
				c_k \e^{2\mathrm{i}(x - x_0^r)z_k+2\mathrm{i} t z_k^2} & 0 \end{pmatrix},
		\end{equation}
		and the residue at $\overline{z}_k$ is  obtained by Schwarz symmetry.
		\end{enumerate}
\end{RHP}

 \begin{theorem}[Inverse problem]\label{thm:inverse}
 	The potential $u(x,t)$ is reconstructed from the unique solution $M(z;x,t)$ of the Riemann--Hilbert problem~\ref{RHP:Mtilde}, with reflection coefficients  $r_1,r_2\in C(\Sigma)$, and  $\rho\in L^{2,2}(\R)\cap L^{1,2}(\R)$,  and discrete spectral data $\{z_k, \zeta(z_k)\}_{j=1}^N$, via
 	\[
 	u(x,t) = 2\mathrm{i} \lim_{z \to \infty} z\, M_{12}(z;x,t).
 	\]
	Furthermore $u(x,t)\in C^2(\R)\times C(\R)$.
 \end{theorem}
 The proof of Theorem~\ref{thm:direct} follows the strategy of  \cite{GJZZ}, with the main additional work concentrated at the real intersection points of $\Sigma$ and $\R$, where the jump matrices have to be reconciled (see Appendix~\ref{a:jumpconsistensy}). Theorem~\ref{thm:inverse} follows from the same arguments as in \cite{GJZZ}; the proof is in Appendix~\ref{a:proof}. The remainder of this introduction is devoted to Theorem~\ref{thm:solgas} below, which shows that the Riemann--Hilbert problem~\ref{RHP:Mtilde} arises as the large-$N$ limit of an $4N$-soliton Riemann--Hilbert problem, namely, as a soliton gas.
 The novelty with respect to previous derivations is that in the limit $N\to\infty$ the  soliton gas solution is non vanishing  at  $x=\pm \infty$ 
 while in \cite{GGJM} and \cite{GGJMM} the soliton gas solution is non vanishing   only for $x=\infty$ or $x=-\infty$.
  This case was first treated in \cite{GJM} for the modified KdV equation.
 \begin{theorem}[Soliton-gas realization]\label{thm:solgas}
 	Consider RHP~\ref{RHP:Mtilde} in the reflectionless case $\rho=0$, with empty discrete spectrum. Then this Riemann--Hilbert problem arises as the large $N$-limit of the $4N$-soliton Riemann--Hilbert problem~\ref{rhp:2N},
 	 whose eigenvalues in $\C^+$ are equally spaced along the two segments $L_\alpha^+$, $\alpha=1,2$ (see figure~\ref{fig:solgas}),
 	  and whose norming constants prescribed by analytic functions $C(z), D(z)$ vanishing at the intersection points $\{\kappa_1,\kappa_2\}=\Sigma\cap\R$. The reflection coefficients of the limit problem are
 	\[
 	r_1(z) = -\mathrm{i} C(z)\,\Phi_\alpha(z), \qquad r_2(z) = \frac{\mathrm{i} D(z)\,\Phi_\alpha(z)^{-1}}{1 + C(z)D(z)} \quad \text{on } L_\alpha^+.
 	\]
where $\Phi_\alpha=\left(\frac{z-\kappa_\alpha}{z-z_\alpha}\right)^{1/2}.$
	 \end{theorem}

 
This manuscript is organized as follows. In Section~\ref{sec_elliptic} we review the solution of the Zakharov--Shabat spectral problem for an elliptic travelling wave. In Section~\ref{sec_perturbative} we formulate the direct scattering problem for initial data of the form \eqref{initial_data} by a perturbative argument and derive the Riemann--Hilbert formulation of the inverse problem. In Section~\ref{sec_soliton} we construct the full soliton gas Riemann--Hilbert problem as the large-$N$ limit of a $4N$-soliton system and prove Theorem~\ref{thm:solgas}. The  proof of Theorem~\ref{thm:inverse}, the non-intersecting configuration, and several technical computations are collected in the Appendix.


\section{Genus-one Riemann-Hilbert Problem}
\label{sec_elliptic}

We consider the genus-one Riemann surface
\be\label{RS_X}
\mathcal{X}=\left\{(z,R)\in\mathbb{C}^2: R^2=(z-z_1)(z-\overline{z_1})(z-z_2)(z-\overline{z_2})\right\}\, .
\ee
We model $\mathcal{X}$ as a two sheeted cover of the Riemann sphere. 
We fix a canonical homology basis on $\mathcal{X}$ by choosing the 
$\mathbf{a}$-cycle to be a closed loop on the first sheet encircling 
the branch points $z_1$ and $\ov{z}_1$ counterclockwise, and the 
$\mathbf{b}$-cycle to be an arc starting from $z_2$ on the first sheet, 
passing to $z_1$, and returning to the starting point on the second 
sheet. See Figure~\ref{surf}. 
The oriented branch cuts used to define $R(z)$ as a single-valued function 
on each sheet will be introduced later. 
%

\begin{figure}[t]
	\centering
	\includegraphics[width=0.5
	\linewidth]{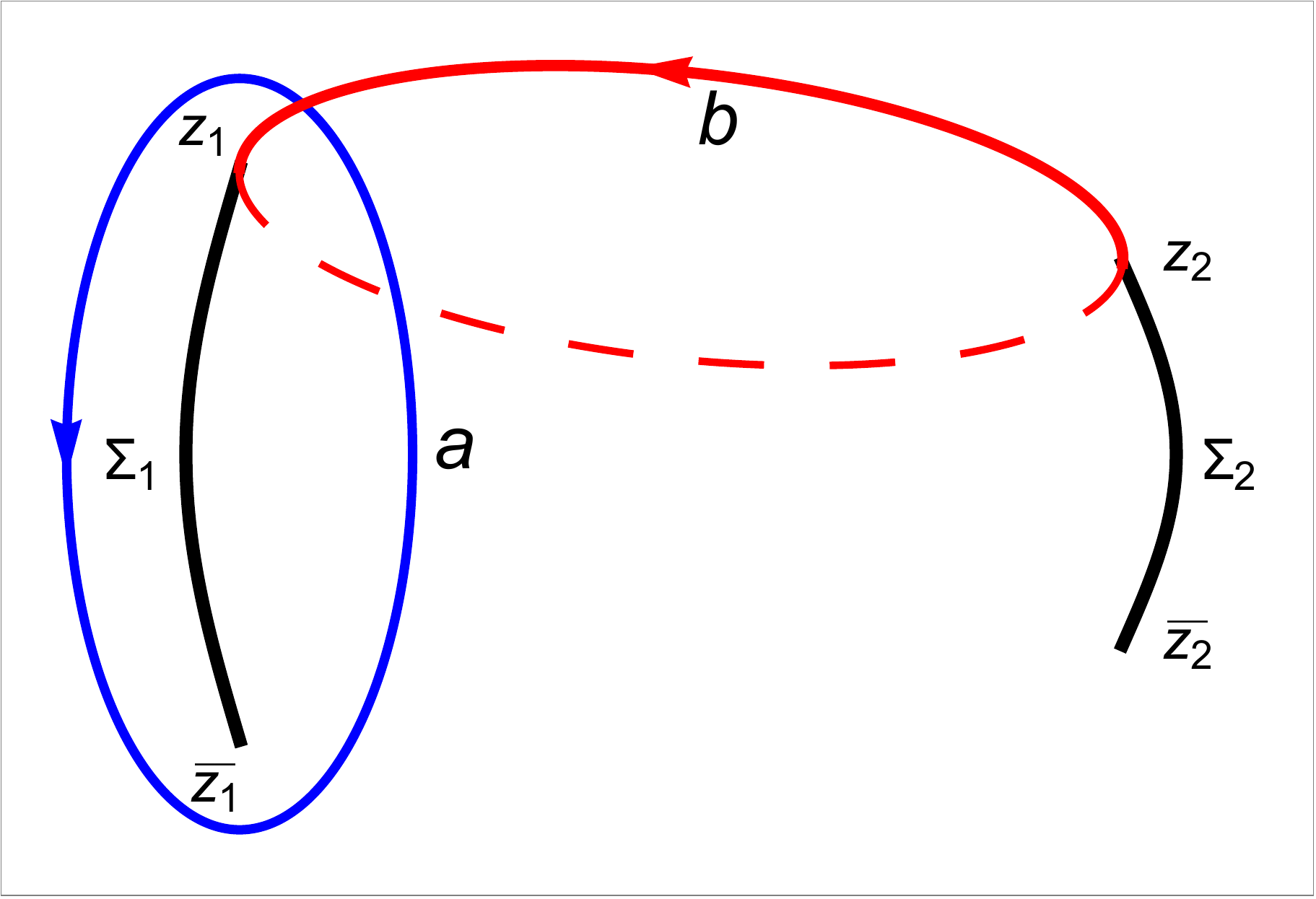}
	\caption{The homology basis for the Riemann surface $\mathcal{X}$ associated with $R^2=(z-z_1)(z-\overline{z_1})(z-z_2)(z-\overline{z_2})$.}
	\label{surf}
\end{figure}

Denote by $\d p$ and $\d q$ the quasi-momentum and quasi-energy differentials, respectively. These are meromorphic differentials of the second kind on $\mathcal{X}$
that are analytic away from the pre-images of $\infty$, where they satisfy 
\begin{align}\label{ap}
	\d p=\pm[1+\mathcal{O}(z^{-2})]\d z,\quad \d q=\pm[2z+\mathcal{O}(z^{-2})]\d z\,, \quad \mbox{as $z\to\infty^\pm$}\,.
\end{align}
and normalized by the condition that 
\begin{align*}
	\oint\limits_{\mathbf{a}}\d p=0\,, \qquad  \oint\limits_{\mathbf{a}}\d q=0\,.
\end{align*}
The above conditions uniquely determine the differentials
\begin{equation}\label{e:quasidiffs}
\begin{gathered}
	\d p = \frac{z^2-\Re(z_1+z_2)z+c_0}{R}\d z\, , \\
	\d q = 2\frac{z^3-\Re(z_1+z_2)z^2+(\Re(z_1)\Re(z_2)+\frac{\Im(z_1)^2+\Im(z_2)^2 }{2})z+c_1}{R}\d z\, ,
\end{gathered}
\end{equation}
where the constant coefficients, computed in \cite{GJZZ}, are given by
\be 
\begin{aligned}
	&	c_0 =\dfrac{1}{2}(|z_1|^2+|z_2|^2)-\frac{1}{2}|z_1-\overline{z_2}|^2\dfrac{E(m)}{K(m)}\, ,\\
	&	c_1= \dfrac{c_0}{2}\Re(z_1+z_2)-\frac{1}{2}(\Re(z_1)|z_2|^2+\Re(z_2)|z_1|^2)\, ,
\end{aligned}
\ee
where $m$ is given in \eqref{e:modulus}. The $\bf{b}$-periods of these differentials, also computed in \cite{GJZZ}, are given by
\be
\label{Omega12}
\Omega_1 :=  \oint_\mathbf{b}\d p= \frac{\pi |z_1- \overline{z_2}|}{K(m)} , \qquad 
\Omega_2 := \oint_\mathbf{b}\d q= - v \Omega_1 \,, \qquad v = -\Re(z_1+z_2).
\ee

%
Choosing $z_2$ as the base point,  
we define the Abelian integrals  $p(z)$ and $q(z)$   by
\be\label{e:defAbelian}
p(z)=\int_{z_2}^z \d p,\quad
q(z)=\int_{z_2}^z \d q.
\ee
According to standard finite-gap theory \cite{Belokolos1994}, expanding these integrals near the point at infinity on the first sheet $z\to\infty^+$ yields:
\be\label{e:pqasymp}
p(z)=z+E+ o(z^{-1}),\quad q(z)=z^2 +N +o(1),
\ee
where 
\begin{equation}\label{e:defEN}
	\begin{split}
		&E:=\lim_{z\to\infty}(p(z)-z)=\lim_{z\to\infty}\int_{z_2}^z[\d p(\lambda)-\d\lambda]-z_2=\dfrac{\int_{\mathbf{a}}\frac{ \lambda \d \lambda}{R}}{\int_{\mathbf{a}}\frac{  \d \lambda}{R}}- \Re(z_1+z_2)\,,\\
		&N:=\lim_{z\to\infty}q(z)-z^2=-vE+\lim_{z\to\infty}(R(z)-z^2-vz)=-vE+\dfrac{v^2}{4}-\dfrac{\omega_0}{2}, 
	\end{split}
\end{equation}
and  $\omega_0=\frac{3}{2}v^2-(z_1(z_2+\overline{z_1}+\overline{z_2})+z_2(\overline{z_1}+\overline{z_2})+\overline{z_1}\overline{z_2})$. Then the fact that $E$ is real follows from the  symmetry  of the curve and the homology basis. 
Writing explicitly $E$  in terms of elliptic integrals we obtain
\begin{equation}
E=-\Re(z_1+z_2)+\overline{z_2}+(\overline{z_1}-\overline{z_2})\dfrac{\Pi(\Lambda,m)}{K(m)}, \qquad \Lambda =\dfrac{z_1-\overline{z_1}}{z_1-\overline{z_2}}
\end{equation}
where $\Pi(n,m)={\displaystyle \int_0^1}\frac{\d s}{(1-n s^2)\sqrt{(1-s^2)(1-ms^2)}}$. 
Let \begin{equation}
	\label{Omega}
	\Omega=\Omega_0+x\Omega_1+t\Omega_2= \Omega_1 ( x - x_0 + \Re( z_1+z_2) t )\,,
\end{equation}
where $\Omega_1$  and $\Omega_2$   have  been defined in \eqref{Omega12}\,,
and $\Omega_0$ is given by 
\begin{equation}
	\label{x0}
	\Omega_0 = - x_0 \Omega_1.
\end{equation}

\begin{figure}[htb]
	\centering
	\begin{overpic}[width=.4\textwidth]{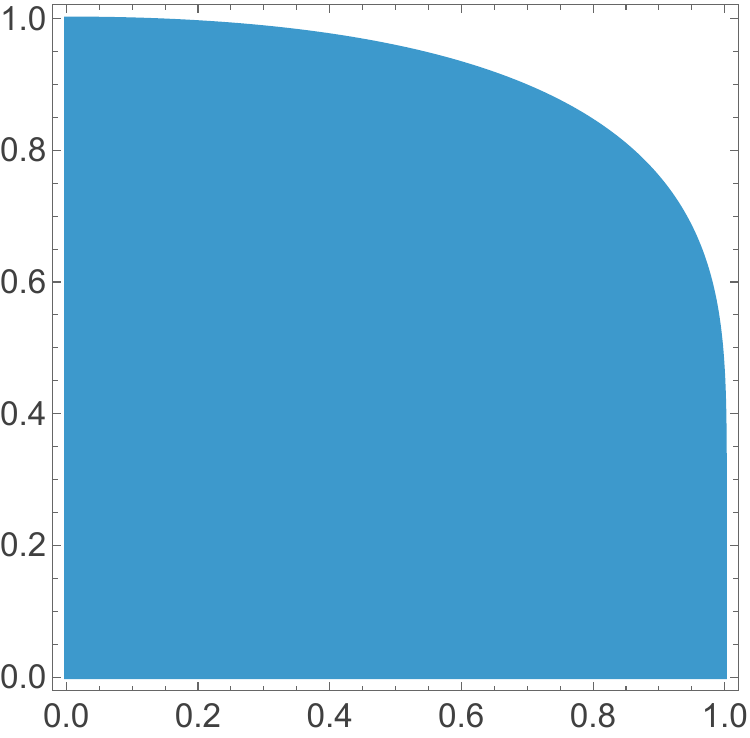}
	\put(51,-3){$m$}
	\put(-3,50){$b$}
	\end{overpic}
	\caption{
	The continuous spectrum of the ZS operator consists of the real axis and arcs $\Sigma_k$, $k=1,2$ connecting $z_k$ to $\overline{z_k}$ across the real axis as shown in Figure~\ref{fig:cuts} when $(m,b)$ is in the shaded subset of $[0,1]^2$. In the unshaded region, the level curve does not cross the real axis; it consists of 
the real axis, an arc connecting $z_1$ to $z_2$, and the complex conjugate arc connecting $\overline{z_1}$ to $\overline{z_2}$. In this paper we consider the shaded region. The unshaded region was studied in \cite{GJZZ}. 
	}
	\label{f:spectral_parameters}
\end{figure}

The continuous spectrum of the background ZS operator consists of the points where $\Im p(z) = 0$. In \cite{GJZZ}, the following facts are established: 
\begin{proposition}\label{prop:sigma}
The existence of two real zeros of the differential $\d p$ defined by \eqref{e:defAbelian} implies that the spectrum of the ZS operator crosses the real line as in Figure~\ref{fig:cuts}.  
The differential $\d p$ has two real zeros if 
\begin{gather}
\label{real_spectrum}
1-m + b - 2\dfrac{E(m)}{K(m)} < 0, 
\end{gather}
where $m$ and $b$ are given by \eqref{e:modulus}.
If the left hand side of \eqref{real_spectrum} is positive, the spectrum of the ZS operator has no real intersections.
The set of values $(m,s)$ where \eqref{real_spectrum} is satisfied is shown in blue in Figure~\ref{f:spectral_parameters}.
\end{proposition}

Throughout the main text of this paper we will assume that branch points $z_1$ and $z_2$ satisfy \eqref{real_spectrum}. 
The case where \eqref{real_spectrum} is not satisfied is considered in Appendix~\ref{s:disjoint}.

We now define the set 
\begin{equation}
	\Sigma := \{ z \in \mathbb{C} \setminus \mathbb{R} 
	\mid \operatorname{Im}(p(z) - E) = 0 \},
\end{equation}
which, because \eqref{real_spectrum} is assumed to be satisfied,  consists of two arcs $\Sigma_1$ and $\Sigma_2$ connecting  $\bar{z}_j$ to $z_j$ for $j = 1, 2$, respectively. See Figure~\ref{fig:cuts}.
We orient each $\Sigma_j$ upward (from $\bar{z}_j$ to $z_j$) and define the sub contours 
\begin{equation}
	\Sigma_j^+ := \Sigma_j \cap \mathbb{C}^+, \qquad 
	\Sigma_j^- := \Sigma_j \cap \mathbb{C}^-,
\end{equation}
and set $\Sigma^\pm := \Sigma_1^\pm \cup \Sigma_2^\pm$. Clearly, $\Sigma_j = \overline{\Sigma_j^+ \cup \Sigma_j^-}$, and $\Sigma := \Sigma^+ \cup \Sigma^-$. 
We model the Riemann surface $\mathcal{X}$ defined by \eqref{RS_X} as a two-sheeted cover of $\C$ cut and glued along $\Sigma$. 
These arcs serve as the branch cuts of the function  $R(z) = \sqrt{(z-z_1)(z-\bar{z}_1)(z-z_2)(z-\bar{z}_2)}$ single-valued on each sheet of 
$\mathcal{X}$, with the sign convention $z^{-2} R(z) \to 1 $ as $z \to \infty$ on the first sheet. 

We define  the  following  Riemann--Hilbert problem.
\begin{RHP}\label{RHP:O}
Find a $2\times2$ matrix valued function $O(z;x,t)$, depending parametrically on $(x,t) \in \R^2$, such that:
	\begin{enumerate}
		\item $O(z;x,t)$ is analytic in $\mathbb{C}\setminus\Sigma$.
		\item $O(z;x,t)=I+\mathcal{O}(z^{-1})$ as $z\to\infty$.
		\item $O(z;x,t)$ satisfies the jump conditions $O_+(z;x,t)=O_-(z;x,t)V^{(O)}(z;x,t)$, where
		\be\label{e:Ojump}
		V^{(O)}(z;x,t)=
		\begin{cases}
			\begin{pmatrix}
				0 & \mathrm{i}\,\e^{2\mathrm{i}\,((x-x_0)E+tN+\varphi_0)}\\
				\mathrm{i}\,\e^{-2\mathrm{i}\,((x-x_0)E+tN+\varphi_0)} & 0
			\end{pmatrix}\,, & z\in \Sigma_2\,,\\
				\begin{pmatrix}
					0 & \mathrm{i} \, \e^{-\mathrm{i}\Omega} \e^{2\mathrm{i}\,((x-x_0)E+tN+\varphi_0)}\\
					\mathrm{i} \, \e^{\mathrm{i}\Omega} \e^{-2\mathrm{i}\,((x-x_0)E+tN+\varphi_0)} &0
				\end{pmatrix}
				\,, & z\in\Sigma_1\,.
		\end{cases}
		\ee
		\item $O(z;x,t)$ admits fourth root singularity at $z\in\{z_1, \overline{z_1},z_2,\overline{z_2}\}$.
		\item $O$ satisfies the symmetry: $\sigma_2 O^*(z)\sigma_2=O(z)$, where $O^*(z;x,t)=\overline{O(\ov{z};x,t)}$, and 
		$
		\sigma_2=\begin{psmallmatrix}
		0 & -\mathrm{i}\\
		\mathrm{i} &0
		\end{psmallmatrix}.
		$
	\end{enumerate}
\end{RHP}
To construct the solution of  RHP \ref{RHP:O}, we introduce the holomorphic differential $\omega$  given explicitly by
\begin{align}\label{e:h.diff}
	\omega= \left( \oint_\mathbf{a} \frac{\d z}{R(z)} \right)^{-1} \frac{\d z}{\,R(z)} = -\i \frac{|z_1-\overline{z_2}|}{4 K(m)} \frac{\d z}{\, R(z)} \,.
\end{align}
We also define the period ratio
\be
\tau=\oint_\mathbf{b}\omega=\i\dfrac{K(1-m)}{K(m)}\,,
\ee
where $m$ was defined in \eqref{e:modulus}.
Next we introduce the Jacobi theta function, \begin{align}
	\theta_3(z;\tau)=\sum_{n\in\mathbb{Z}}\e^{2\i\pi nz+\i\pi n^2\tau},\quad z\in\mathbb{C},
\end{align}
that  satisfies the periodicity relations
\begin{align}
	\theta_3(z+h+k\tau;\tau)=\e^{-\i\pi k^2\tau-2\i\pi kz}\theta_3(z;\tau),\quad h,k\in\mathbb{Z}.
\end{align}
We also recall that the Jacobi elliptic function  vanishes on the
half period $\frac{\tau}{2}+\frac{1}{2}$.
Using $\omega$, we define the Abel integral with base point $z_2$
\begin{align}\label{Abel}
	A(z)=\int_{z_2}^{z}\omega,
\end{align}
where the path of integration avoids all the arcs $\Sigma_1\cup \Sigma_2$. We observe that
\begin{align}
	\quad A(z_{1,+})=\frac{\tau}{2},\quad A(\overline{z}_{1,+})=\frac{1}{2}+\frac{\tau}{2},\quad A(\overline{z_2})=\frac{1}{2}.
\end{align}
In addition, the following jump relations hold across the cuts:
\begin{align}
	& A(z_+)+A(z_-)=0,\quad z\in\Sigma_2\,,\\
	& A(z_+)+A(z_-)=\tau,\quad z\in\Sigma_1\,,
\end{align}
where here and below $A(z_\pm)$ denote the boundary values of $A(z)$ as $z$ approaches the oriented curves $\Sigma_1\cup\Sigma_2$ from the left ($+$) and right ($-$) sides of the orientation.

\begin{proposition}\label{prop:Osol}
RHP \ref{RHP:O} has a unique solution which can be expressed explicitly in the form:
\begin{align}\label{defO}
	O(z;x,t) = 
	\e^{\i((x-x_0)E + t N+\varphi_0)\sigma_3}
	\begin{bmatrix}
		\frac{\gamma+\gamma^{-1}}{2} H_{11}(z) & \frac{\gamma-\gamma^{-1}}{2} H_{12}(z) \\
		\frac{\gamma-\gamma^{-1}}{2} H_{21}(z) & \frac{\gamma+\gamma^{-1}}{2} H_{22}(z)
	\end{bmatrix}
	\e^{-\i((x-x_0)E + t N+\varphi_0)\sigma_3}\, ,
\end{align}
where $\gamma= \gamma(z)$ is the analytic function in $\C\setminus\Sigma$ defined by
\be\label{e:gamma}
\gamma(z)=\left(\frac{(z-z_1)(z-z_2)}{(z-\overline{z_1})(z-\overline{z_2})}\right)^{1/4},
\ee
normalized such that $\gamma(z)\to1$ as $z\to\infty$, and satisfying the jump condition
\be
\gamma_+(z)=\mathrm{i}\gamma_-(z),\qquad z\in\Sigma_1\cup\Sigma_2\,.
\ee
The matrix entries $H_{i j}$ $(i,j=1,2)$  are expressed as
\bse\label{e:H.theta}
\begin{gather}
	H_{11}(z) = \frac{\theta_4(0)\theta_4(A(z)-A(\infty)-\frac{\Omega}{2\pi})}{\theta_4(\frac{\Omega}{2\pi})\theta_4(A(z)-A(\infty))}, \quad
	H_{12}(z) = \frac{\theta_4(0) \theta_4(A(z)+A(\infty)+\frac{\Omega}{2\pi})}{\theta_4(\frac{\Omega}{2\pi}) \theta_4(A(z)+A(\infty))}
	\\
	H_{21}(z) = \frac{\theta_4(0)\theta_4(A(z)+A(\infty)-\frac{\Omega}{2\pi})}{\theta_4(\frac{\Omega}{2\pi})\theta_4(A(z)+A(\infty))}
	, \quad
	H_{22}(z) = \frac{\theta_4(0)\theta_4(A(z)-A(\infty)+\frac{\Omega}{2\pi})}{\theta_4(\frac{\Omega}{2\pi})\theta_4(A(z)-A(\infty))}.
\end{gather}
\ese
Here $\theta_4(z; \tau) = \theta_3(z+\tfrac{1}{2}; \tau)$ is another Jacobi theta function with a half-period shift relative to $\theta_3(z; \tau)$.

\end{proposition}
From these  matrices  we can construct the solution of the ZS linear spectral problem \eqref{e:zs} for an elliptic travelling wave potential $u_0$ as follows. 
\begin{proposition}\label{prop:W0}
Let $O(z;x,t)$ be the matrix defined in \eqref{defO}, and $p(z)$ and $q(z)$ be the Abelian integrals  introduced in \eqref{e:defAbelian}. Define the matrix-valued function 
	\be\label{defWbg}
W_e(x,t;z)=O(z;x,t)\e^{-\mathrm{i}\left((x-x_0)(p(z)-E)+t(q(z)-N)\right)\sigma_3}\,.
	\ee
Then $W_e(x,t;z)$ is a simultaneous solution to the Lax pair \eqref{e:NLSLP} associated with the elliptic traveling wave potential 
\be
	\label{elliptic_general}
	u_e(x,t; x_0,\varphi_0)=\Im(z_2+z_1) \frac{\theta_4(0) \theta_4(2A(\infty)+\frac{\Omega}{2\pi})}{\theta_4(\frac{\Omega}{2\pi}) \theta_4(2A(\infty))}\e^{2\mathrm{i}((x-x_0)E+tN+\varphi_0)},
\ee
where the phase $\Omega=\Omega_1 ( x - x_0 + \Re( z_1+z_2) t )$ with $\Omega_1$ given by \eqref{Omega12}, the real parameters $E $, $N$,  and the spatial shift $x_0$ are specified in  \eqref{e:defEN} and \eqref{x0}, respectively.
The expression \eqref{elliptic_general} for the elliptic potential solution of NLS can be reduced to the Jacobi elliptic formula \eqref{e:elliptic0}.
\end{proposition}	
The proof of Proposition~\ref{prop:W0} follows the argument in \cite{GJZZ} with one modification: with the choice of branch cuts made in this paper, the Abelian integrals satisfy $p(z_{+})+p(z_{-})=0$ and $q(z_{+})+q(z_{-})=0$ for $z\in\Sigma_2$; while $p(z_{+})+p(z_{-})=\Omega_1$ and $q(z_{+})+q(z_{-})=\Omega_2$ for $z\in\Sigma_1$. 
The jump condition for $W_e(x,t ;z)$ is then given by 
	\begin{align}\label{e:jumpW0}
		\begin{split}
			W_e(x,t;z_+)&=W_e(x,t;z_-)\e^{\i\left((x-x_0)(p(z_-)-E)+t(q(z_-)-N)\right)\sigma_3}V^{(O)} \e^{-\i\left((x-x_0)(p(z_+)-E)+t(q(z_+)-N)\right)\sigma_3}\\
			&=W_0(z_-)
				\i \e^{2\i \varphi_0\sigma_3}\sigma_1, \quad z \in \Sigma_1\cup \Sigma_2.				
					\end{split}
	\end{align}
	The jump of $W_e$ is independent of $z$, $x$ and $t$. A standard dressing argument considering the entire functions $(\partial_x W_e) W_e^{-1}$ and $(\partial_t W_e) W_e^{-1}$ 
	shows that $W_e(x,t;z)$ is a simultaneous solution of the Lax pair \eqref{e:NLSLP}. Formula \eqref{elliptic_general} follows from recovering the corresponding solution of \eqref{e:nls} from the formula
	$u_e(x,t;x_0,\varphi_0) = \lim_{z \to \infty} 2\i z (W_e)_{12}(z;x,t) = \lim_{z \to \infty} 2\i z O_{12}(z;x,t)$.

\section{Perturbative Argument}
\label{sec_perturbative}
Consider the ZS  problem \eqref{e:zs} for a steplike elliptic initial data $u_0$ of the form \eqref{initial_data}.  
Our goal is to determine the  Jost solutions   $W^r(x;z)$ and $W^\ell(x;z)$ of \eqref{e:zs} satisfying the specific boundary conditions:
\bse \label{e:Wpmasy}
\begin{align}
 W^r(x;z)&=O^r(x,0;z)(I+\mathcal{O}(x^{-1}))\e^{-\i (x-x_0^r)(p(z)-E)\sigma_{3}}, \quad \mbox{as $x\to  \infty$},\\
 W^{\ell}(x;z)&=O^\ell(x,0;z)(I+\mathcal{O}(x^{-1}))\e^{-\i (x-x_0^{\ell})(p(z)-E)\sigma_{3}}, \quad \mbox{as $x\to-\infty$}.
 \end{align}
 \ese
To factor out the asymptotic oscillations, we introduce the modified Jost solutions  $m^s(x;z)$ for the ZS equation \eqref{e:zs} through the transformation:
\begin{align}\label{e:defm}
W^s(x;z)=O^s(z;x,0) m^s(x;z)\e^{-\i (x-x_0^s)(p(z)-E)\sigma_{3}},\quad s\in\{\ell,r\}\,.
\end{align}
Note that the matrix-valued function $O^s(x, z)$ is uniformly boundeded with respect to $x$ for any spectral parameter $z \in \mathbb{C}\setminus\partial \Sigma$. 
It is convenient to introduce the perturbation matrices $\Delta U^s(x)$, which measure the deviation of the potential $u_0(x)$ from the background $u_e^s(x)$:
\begin{align*}
	\Delta U^s(x):=\begin{pmatrix}0 & u_0(x)-u^s_{e}(x) \\ -\left(\overline{u_0(x)}-\overline{u^s_{e}(x)}\right) & 0\end{pmatrix}\,,\qquad s\in\{\ell,r\}\,.
\end{align*}
\begin{lemma} 
Assume that $W^s(x;z)$ is the solution to the ZS spectral problem \eqref{e:zs} associated with the potential $u_0(x)$ .
Then the modified Jost function $m^s(x;z)$ defined in  \eqref{e:defm} satisfies the following ODE
\begin{align}\label{e:ODEm}
\partial_{x} m^s=-\mathrm{i} (p(z)-E)\left[\sigma_{3}, m^s\right]+O^s(z;x,0)^{-1} \Delta U^s(x) O^s(z;x,0) m^s, 
\end{align}
where $[a,b]=ab-ba$ denotes the matrix commutator.
\end{lemma}
The proof closely parallels that of Lemma 3.1 in \cite{GJZZ} and is thus omitted for brevity. 
Integrating the differential equation \eqref{e:ODEm}  yields the corresponding Volterra integral equation:
\be\label{e:Intm}
m^s(x;z)=I+\int_{\infty^s}^{x} \e^{-\i (x-y) (p(z)-E) \sigma_{3}} O^s(y;z)^{-1} \Delta U^s(y)O^s(y;z) m^s(y;z)\e^{\i (x-y) (p(z)-E) \sigma_{3}}  \d y\,,
\ee
where $\infty^\ell$ and $\infty^r$ denote $-\infty$ and $+\infty$, respectively.

We now recall the following property of the Abelian integral $p(z)$ of the quasi-momentum differential $\d p$, established in \cite{GJZZ}. It will be used below to prove the analyticity and boundedness properties of $m^s(x;z)$ and $W^s(x;z)$.
\begin{lemma}
	\label{dp_inequality}
	For any choice of branch points $z_1, z_2 \in \C^+$ with $z_1 \neq z_2$, the Abelian integral $p(z)$ defined by \eqref{e:defAbelian} satisfies the inequalities
	\be\label{e:imP}
	\begin{cases}
		\Im (p(z)-E)>0\,, & z\in \mathbb{C}^+\setminus\Sigma^+\,,\\
		\Im (p(z)-E)<0 \,, & z\in \mathbb{C}^{-}\setminus\Sigma^-\,.\\
	\end{cases}
	\ee
\end{lemma}

In what follows $W_i^{s}(x; z)$ denotes the $i$-th column of $W^s(x; z)$. The main properties of the Jost solutions $W^{s}(x; z)$ are summarized in the following proposition.
\begin{proposition}\label{propm}
Suppose $u_0-u^\ell_e \in L^1(\Real^-)$  and $u -u^r_e \in L^1(\Real^+)$. Then $W^s(x;z)$, $s \in \{ \ell, r \}$, have the following properties:
\begin{enumerate}
	\item For every $z_\pm \in \R \cup \operatorname{int}(\Sigma)$\footnote{Here $z_\pm \in \operatorname{int}(\Sigma)$ denotes $z$ as a left/right boundary values on the interior of $\Sigma^s$.}, and each $x_*\in\Real$, there exist unique solutions $W^\ell(\, \cdot\, ;z)\in L^\infty(-\infty,x_*)$ and $W^r(\, \cdot\, ;z)\in L^\infty(x_*,\infty)$ of \eqref{e:zs} given by \eqref{e:defm}, where $m^s(x;z)$, $s\in\{\ell,r\}$, are solutions of the integral equations \eqref{e:Intm}. 
	Moreover, for each fixed $x\in \R$, the columns $W_1^{\ell}(x;z)$ and $W_2^r (x;z)$  are analytic for $z\in\mathbb{C}^+\setminus\Sigma^{+}$; 
	 $W_1^r(x;z)$ and $W_2^{\ell}(x;z)$ are analytic for $z\in\mathbb{C}^-\setminus\Sigma^{-}$.
	\item For $s\in\{\ell,r\}$ and $z \in \Sigma^s$, the boundary values $W^s(x;z_\pm)$ satisfy 
	\begin{equation}\label{jumpW}
		W^s(x;z_+)= W^s(x;z_-)\e^{\mathrm{i} \varphi_0^s \sigma_3} (\mathrm{i} \sigma_1),\quad z\in\Sigma_1\cup\Sigma_2\,.
	\end{equation}
	\item The column vectors satisfy the symmetries
	\begin{equation}\label{schwarzW}
	\begin{aligned}
		&W_1^\ell(x;z) = \mathrm{i} \sigma_2 \overline{W_2^\ell(x; \bar{z})}, \quad
		&&W_2^r(x;z) = - \mathrm{i} \sigma_2 \overline{W_1^r(x; \bar{z})}, \qquad
		& z \in \C^+ \setminus \Sigma^+, \\
		&W_1^r(x;z) =  \mathrm{i} \sigma_2 \overline{W_2^r(x; \bar{z})}, 
		&&W_2^\ell(x;z) = - \mathrm{i} \sigma_2 \overline{W_1^\ell(x; \bar{z})}, 
		& z \in \C^- \setminus \Sigma^- ,\\
	\end{aligned}
	\end{equation}
	with continuous extensions to the boundary $\R \cup \Sigma$. 
\end{enumerate}
\end{proposition}

\subsection{Symmetry of solutions and continuous spectral data}
For $z\in\mathbb{R}\cup\Sigma$, $W^{\ell}(z)$  and  $W^r(z)$ are fundamental  matrix solutions of the ZS equation \eqref{e:zs}. Therefore, there exists a scattering matrix
\be\label{e:S}
S(z)=\begin{pmatrix}a(z) & -b^*(z) \\ b(z) & a^*(z)\end{pmatrix}, \quad z \in \mathbb{R};\qquad
S(z_{\pm})=\begin{pmatrix}a(z_{\pm}) & -b^*(z_{\pm}) \\ b(z_{\pm}) & a^*(z_{\pm})\end{pmatrix}, \quad z_{\pm} \in \Sigma\,,
\ee
such that 
\bse
\begin{gather}
W^{\ell}(x; z)=W^r(x; z) S(z),\quad z \in \mathbb{R}, \\
\label{e:S.Sigma}
W^{\ell}(x ; z_\pm)=W^r(x ; z_\pm) S(z_\pm),\quad z\in\Sigma.
\end{gather}
\ese
The analyticity and large-$z$ behavior of the scattering coefficients are obtained in the same way as in our previous paper~\cite{GJZZ}. Indeed, these properties depend only on the Volterra equations for the modified Jost functions $m^s(x;z)$ and on the decay assumptions on the perturbation, and are not affected by the fact that the spectral bands $\Sigma$ intersect the real axis. We recall the result for completeness.
\begin{theorem}\label{theorem1}
Suppose that $u_0 -u_e^\ell \in L^1(\Real^-)$ and $u_0 -u_e^r \in L^1(\Real^+)$. Then the scattering data $a(z)$ and $b(z)$ defined in \eqref{e:S} satisfy the following properties:
    \begin{enumerate}
        \item The scattering coefficients can be expressed in terms of $W^\pm(x;z)$ as
 \be \label{e:scoefs}
         a(z)=\det[W_1^{\ell}(x;z), W_2^r(x;z)],\qquad b(z)=\det[W^r_1(x;z),W^{\ell}_1(x;z)]\,.
        \ee
        It follows that $a(z)$ is analytic for $\,$ $\C^+\setminus\Sigma^+$. Moreover, both $a(z)$ and $b(z)$ have at worst square  root singularities at $\{z_1, z_2, \overline{z_1}, \overline{z_2}\}$.

        \item If   $u_0 -u_e^\ell \in \mathcal{W}^{1,1}(\Real^-)$ and $u_0-u_e^r \in \mathcal{W}^{1,1}(\Real^+)$, then for $z\in\overline{\C^+}$,
        \bse \label{e:scat.asympt}
        \be
         \label{a+asympt}
         \lim_{z\to\infty}a(z)\e^{-\mathrm{i} (x_0^{\ell} -x_0^r)z}=1+\mathcal{O}(z^{-1}),
         \ee
         \ese
         Moreover, if $u_0-u_e^\ell \in \mathcal{W}^{4,1}(\Real^-)$ and $u_0 -u_e^r \in \mathcal{W}^{4,1}(\Real^+)$, then
         for $z\in\Real$,
        \begin{align}\label{b+asympt}
         b(z)\e^{-\mathrm{i}(x_0^{\ell} -x_0^r)z}=\mathcal{O}(z^{-4}), \qquad |z|\to\infty.
         \end{align}
         \end{enumerate}
        \end{theorem}

The next proposition records the jump relations satisfied by the scattering coefficients across the bands $\Sigma$. These relations follow directly from the jump conditions of the Jost solutions and will be used below to construct the scalar function $h$ and the factorization of $a(z)$.
\begin{proposition}\label{S}
    Let $u_0 -u_e^\ell \in L^1(\Real^-)$, $u_0 -u_e^r \in L^1(\Real^+)$,  and $a(z)$ and $b(z)$ be the scattering data in \eqref{e:S}, then
for $z \in \Sigma$, the scattering data satisfy the jump relations
    \begin{equation}\label{jumpS}
        \begin{bmatrix}
           a(z_+) & -b^*(z_+) \\ b(z_+) & a^*(z_+)
        \end{bmatrix}
        =
      \begin{bmatrix}
      	a^*(z_-)\e^{\mathrm{i}(\varphi_0^r-\varphi_0^\ell)} & 
      	b(z_-)\e^{\mathrm{i}(\varphi_0^r+\varphi_0^\ell)} \\ 
      	-b^*(z_-)\e^{-\mathrm{i}(\varphi_0^r+\varphi_0^\ell)} & 
      	a(z_-)\e^{-\mathrm{i}(\varphi_0^r-\varphi_0^\ell)}
      \end{bmatrix}, \quad z\in \Sigma.
    \end{equation}
\end{proposition}

\begin{proof}
	The jump relations \eqref{jumpS} on the scattering coefficients are a direct consequence of the jump relations satisfied by the Jost functions \eqref{jumpW}, the relations \eqref{e:scoefs} defining the scattering coefficients, and the Schwarz symmetry \eqref{schwarzW}. We only give the calculation for $a(z)$:
\begin{multline*}
	a(z_+) = \det[ W_1^\ell(x;z_+), \ W_2^r(x,z^+)] 
	= \e^{\i (\varphi_0^r-\varphi_0^\ell)} \det \left[W_2^\ell (x;z_-), \ W_1^r (x;z_-) \right] \\
	= \e^{\i (\varphi_0^r-\varphi_0^\ell)} \det \left[ -\i \sigma_2 {W_1^\ell}^* (x;z_-), \ \i \sigma_2 {W_2^r}^* (x;z_-) \right] 
	= \e^{\i (\varphi_0^r-\varphi_0^\ell)}  \det\left[ {W_1^\ell}^* (x;z_-), \ {W_2^r}^* (x;z_-) \right] \\
	= \e^{\i (\varphi_0^r-\varphi_0^\ell)} a^*(z_-).
\end{multline*} 
The other calculations are similar. 
\end{proof}

\subsection{Discrete spectral data}
For a given step-like potential $u_0$, consider the zero set 
\begin{equation}\label{e:azero}
	\zeroset (u_0)  = \zeroset = \{ z \in \overline{\mathbb{C}^+} \, : \, a(z) = 0 \}. 
\end{equation}
At any point $z_0 \in \zeroset$, due to the determinantal formula \eqref{e:scoefs}, the two vector solutions $W_1^\ell(x; z_0)$ and $W^r_2(x; z_0)$ are linearly dependent. So, for each $z_0 \in \zeroset$ there exists a nonzero complex number $\zeta = \zeta (z_0)$ such that 
\begin{subequations}
\begin{gather}\label{Darboux.const}
	W_1^\ell(x; z_0) = \zeta(z_0)  W_2^r(x; z_0), \qquad \zeta:\zeroset \to \C \setminus \{0\}, \quad z_0 \in \zeroset.
\shortintertext{The Schwarz symmetry \eqref{schwarzW} implies that $a^*(z_0) = 0$ and}
	W_2^\ell(x, \overline{z_0}) = - \overline{\zeta (z_0)} W_1^r(x; \overline{z_0}), \qquad z_0 \in \zeroset.
\end{gather} 
\end{subequations}
Notice that for $z_0 \not\in \Sigma \cap \R$, that is, not in the continuous spectrum,  the two vector solutions $W_1^\ell(x;z_0)$ and $W_2^r(x;z_0)$ defined by \eqref{e:defm}-\eqref{e:Intm} decay exponentially to zero, respectively, as $x \to -\infty$ and $x \to + \infty$. It follows that any $z_0 \in \zeroset \setminus (\Sigma \cap \R)$ is a true $L^2$ eigenvalue of the Zakharov-Shabat scattering operator \eqref{e:zs}.  Because the scattering operator is non self-adjoint, the structure of $\zeroset$ can be quite complicated. Generally speaking $\zeroset$ can contain zeros of higher order, can intersect the continuous spectrum, and possess accumulation points in the continuous spectrum. Here we assume that: 
\begin{enumerate}[label = (\arabic*)]
	\item The zeros of $a(z)$ are all simple and isolated. 
	\item There are no zeros of $a(z)$ on the real line, i.e., $\zeroset \cap \R = \emptyset$. 
\end{enumerate}
These conditions are generic in the sense that they define a dense open subset of the collection of all allowable initial data. 
Since the asymptotic condition \eqref{a+asympt} ensures that $\zeroset$ is bounded, one immediate consequence of these assumptions is that $\zeroset$ is finite. So, under these assumptions we can enumerate the discrete spectrum $\zeroset = \{ z_k \}_{k=1}^N$. 
We allow for the possibility that $\zeroset \cap \Sigma \neq \emptyset$. However, because $a(z)$ is not analytic for $z \in \Sigma$, any point $z _0 \in \Sigma \cap \zeroset$ must be considered as a boundary value of the cut plane $\C \setminus \Sigma$. Knowing that $\lim_{z \to z_0 \in \Sigma_+} a(z) = 0$ does not imply that $\lim_{z \to z_0 \in \Sigma_-} a(z) = 0$. Here we remind the reader that our convention is that subscripts $\pm$ refer to boundary values along the curve; superscripts $\pm$ denote the parts of $\Sigma$ in the upper/lower complex half-planes. 

\subsection{The inverse scattering problem}
To set up a Riemann--Hilbert problem, we introduce the sectionally meromorphic matrix $M_o(z;x)$:
\be \label{WB}
M_o(z;x)=
\begin{cases}
\left[\dfrac{W_{1}^{\ell}(x; z)}{a(z)}, \;W_{2}^r(x;z)\right] \e^{\i (x-x_0^r)z\sigma_3}, & z\in\Complex^+\setminus \Sigma^+\,,\vspace{5pt}\\ 
\left[W^r_1(x;z), \dfrac{W^{\ell}_2(x;z)}{a^*(z)}\right] \e^{\i(x-x_0^r)z\sigma_3}, & z\in\Complex^-\setminus\Sigma^-\,.
\end{cases}
\ee

Then $M_o(z;x)$ satisfies the following Riemann--Hilbert problem:
\begin{RHP}\label{RHP:M}
 Find a $2\times2$ matrix-valued function $M_o(z;x)$ which satisfies the following conditions:
 \begin{enumerate}
 	\item $M_o(z;x)$ is analytic  in $\Complex\setminus (\Real\cup\Sigma)$.
    \item $M_o(z;x)=I+\mathcal{O}(z^{-1})$, as $z\to\infty$.
    \item $M_o(z;x)$ has simple poles at each $z_k \in Z$ and 
    $\ov{z}_k \in \ov{Z}$, with residues
    \begin{equation}
    	\underset{z = z_k}{\operatorname{Res}}\, M_o(z;x) 
    	= \lim_{z \to z_k} M_o(z;x) 
    	\begin{pmatrix} 0 & 0 \\ 
    		c_k \e^{2\i(x - x_0^r)z_k} & 0 \end{pmatrix}, \qquad c_k = \frac{\zeta(z_k)}{a'(z_k)}
    \end{equation}
    and by Schwarz symmetry
    \begin{equation}
    	\underset{z = \ov{z}_k}{\operatorname{Res}}\, M_o(z;x) 
    	= \lim_{z \to \ov{z}_k} M_o(z;x) 
    	\begin{pmatrix} 0 & -\overline{c_k} 
    		\e^{-2\i(x - x_0^r)\ov{z}_k} \\ 
    		0 & 0 \end{pmatrix},
    \end{equation}
    where $\zeta(z_k)$ is the norming constant defined in \eqref{Darboux.const}.
    \item $M_o(z;x)$ admits the following jump condition:
    \be\label{jumpM}
    M_o(z_+;x)=M_o(z_-;x)\begin{cases}
    \begin{pmatrix}
    	-\dfrac{\mathrm{i} b^*(z_-)}{a(z_+)}\e^{-\mathrm{i} \varphi_0^\ell } & \mathrm{i}\e^{-2\mathrm{i} (x-x_0^r)z+\mathrm{i} \varphi_0^r } \\[1em]
    	\dfrac{\mathrm{i} \e^{2\mathrm{i} (x-x_0^r)z-\mathrm{i}\varphi_0^\ell }}{a(z_+)a(z_-)} & -\dfrac{\mathrm{i} b(z_-)}{a(z_-)}\e^{\mathrm{i} \varphi_0^r}
    \end{pmatrix}, & z\in\Sigma^+,
    	\vspace{6pt}
    	\\[1em]
    	\begin{pmatrix}
    		\dfrac{1}{|a(z)|^2} & \dfrac{b^*(z)}{a^*(z)}\e^{-2\mathrm{i}(x-x_0^r)z}\\[1em]
    		\dfrac{b(z)}{a(z)}\e^{2\mathrm{i}(x-x_0^r)z} &1
    	\end{pmatrix}, & z\in \Real,
    	\vspace{6pt}
    	\\
    	\begin{pmatrix}
    		\dfrac{\mathrm{i} b^*(z_-)}{a^*(z_-)} \e^{-\mathrm{i} \varphi_0^r}& \dfrac{\mathrm{i} \e^{-2\mathrm{i}(x-x_0^r)z+\mathrm{i} \varphi_0^\ell }}{a^*(z_+)a^*(z_-)}\\[1em]
    		\mathrm{i} \e^{2\mathrm{i}(x-x_0^r)z-\mathrm{i} \varphi_0^r} & \dfrac{\mathrm{i} b(z_-)}{a^*(z_+)}\e^{\mathrm{i} \varphi_0^\ell }
    	\end{pmatrix}, & z\in \Sigma^{-}\,.
    \end{cases}
    \ee
     \item 
     $M_o(z;x)$ satisfies Schwarz symmetry:
     \be
     M_o(z;x)=\sigma_2 \overline{M_o(\,\overline{z};x)} \sigma_2.
     \ee
     \item 
     $M_o(z;x)$ admits quartic root singularites at $z\in\{z_1, \overline{z_1},z_2,\overline{z_2}\}$.
\end{enumerate}
\end{RHP}
The jump consistency at the intersection points of $\Sigma$ and $\mathbb R$ requires a separate verification. Indeed, the construction of $M_o$ differs from that in our previous work~\cite{GJZZ}, because in the present setting the spectral bands $\Sigma$ are allowed to intersect the real axis. Hence one has to check that the jump matrices defined on $\Sigma$ and on $\mathbb R$ are compatible at the intersection points. This verification is given in Appendix~\ref{a:jumpconsistensy}.
\begin{proposition}\label{p:hep}
	Define the function  $h(z)$  by
	\begin{equation}\label{e:defh}
		\begin{aligned}
			h(z)=\exp \left\{  \frac{1}{2\pi\i}  \left( 
			\int_{\Sigma^+}
			\frac{\log\left(-\frac{b(s_+)}{b(s_-)}\right)}{s-z} \d s
			+\int_{\Sigma^-}
			\frac{\log\left(-\frac{b^*(s_-)}{b^*(s_+)}\right)}{s-z} \d s 
			-\int_{\Real}\frac{\log(1+|\rho(s)|^2)}{s-z} \d s
			\right) \right\} \,.
		\end{aligned}
	\end{equation} 
where  $\rho(z)=\frac{b(z)}{a(z)}$ is the reflection coefficient. Then, $h(z)$ 
	satisfies the following properties:
	\begin{itemize}
		\item $h(z)$ is analytic in $\C \setminus (\R \cup \Sigma)$ and satisfies the symmetry $h^*(z)  h(z) =1$.
		\item For $z \in \R \cup \operatorname{int}( \Sigma)$ the boundary values of $h(z)$ satisfy the jump relations:
		\be\label{e:jumph}
		\frac{h(z_+)}{h(z_-)}=
		\begin{cases}
			-\dfrac{b(z_+)}{b(z_-)}, & z\in \Sigma^+,\\[0.8em]
		-\dfrac{b^*(z_-)}{b^*(z_+)}, & z\in \Sigma^-,\\[0.8em]
			\left(1+\left|r(z)\right|^2\right)^{-1}, & z\in \mathbb{R}.
		\end{cases} 
		\ee
		\item $h(z) = 1 + \bigo{z^{-1}}$ as $z \to \infty$. 
		\item $h(z)$ is bounded as $z$ approaches any endpoint of $\Sigma$ 
		or any intersection point of $\Sigma$ with $\Real$.
	\end{itemize}
	\end{proposition} 
	\begin{proof}
		We present the details as $z\to z_2$, $z\to \ov{z}_2$ and $z\to \xi_2$. The calculation is similar at the points on $\Sigma_1$.
		As $z\to z_2$ along 
		\be
			\lim_{\substack{s\to z_2\\ s\in \Sigma_2^+}}
			\frac{b(s_+)}{b(s_-)}
			=
			\lim_{\substack{s\to z_2\\ s\in \Sigma_2^+}}
			\frac{\gamma(s_+)^{-2}}{\gamma(s_-)^{-2}}\,
			\frac{\det\!\bigl[\,\gamma(s_+)W_1^r(x;s_+),\,\gamma(s_+)W_1^\ell(x;s_+)\,\bigr]}
			{\det\!\bigl[\,\gamma(s_-)W_1^r(x;s_-),\,\gamma(s_-) W_1^\ell(x;s_-)\,\bigr]}
			=-1.
		\ee
		As $z\to\xi_2$, both the $\Sigma_2^+$- and $\Sigma_2^-$-integral contribute, since from \eqref{jumpS}, we have $\frac{b(s_+)}{b(s_-)}=\frac{b^*(s_-)}{b^*(s_+)}$ for $s \in \Sigma_2$. 
		It follows that the integrals along $\Sigma_2^+$ and $\Sigma_2^-$ in \eqref{e:defh} can be expressed as
		\be
			\int_{\Sigma_2^+}
		\frac{\log\left(-\frac{b(s_+)}{b(s_-)}\right)}{s-z} \d s
		+\int_{\Sigma_2^-}
		\frac{\log\left(-\frac{b^*(s_-)}{b^*(s_+)}\right)}{s-z} \d s =\int_{\Sigma_2}	\frac{\log\left(-\frac{b(s_+)}{b(s_-)}\right)}{s-z} \d s.
		\ee
		Since $b$ is smooth, the boundary values of the integral at $\xi_2$, an interior point of $\Sigma_2$, are bounded.  
	\end{proof}
Since $a(z)$ has finitely many simple zeros $Z = \{z_k\}_{k=1}^N \subset \mathbb{C}^+$ in \eqref{e:azero}, we introduce the Blaschke product
\begin{equation}\label{e:Blaschke}
	\mathcal{B}(z) = \prod_{k=1}^{N} \frac{z - z_k}{z - \bar{z}_k}\,,
\end{equation}
which satisfies $|\mathcal{B}(z)| = 1$ for $z \in \mathbb{R}$ and 
$\mathcal{B}(z) \to 1$ as $z \to \infty$ in $\mathbb{C}^+$. We then define the zero-free factor
\begin{equation}\label{e:a0}
	a_0(z) = \frac{a(z)}{\mathcal{B}(z)}\,, \qquad z \in \mathbb{C}^+ \setminus \Sigma^+\,.
\end{equation}
By construction, $a_0(z)$ is analytic and nonvanishing in $\mathbb{C}^+ \setminus \Sigma^+$, and satisfies
\begin{equation}\label{e:a0asymp}
	a_0(z)\, \e^{-\i(x_0^\ell - x_0^r)z} = 1 + O(z^{-1})\,, \qquad z \to \infty\,.
\end{equation}
We now introduce a factorization of $a(z)$ into $a_1(z)$ and $a_2(z)$, defined by
\begin{equation}\label{e:defa12}
	a_1(z) = \bigl(a_0(z)\, h(z)\bigr)^{1/2}\, 
	\e^{\frac{\i}{2}(x_0^\ell - x_0^r)z}\, \mathcal{B}(z)\,,
	\qquad
	a_2(z) = \left(\frac{a_0(z)}{h(z)}\right)^{1/2} 
	\e^{-\frac{\i}{2}(x_0^\ell - x_0^r)z}\,,
	\qquad z \in \mathbb{C}^+\,.
\end{equation}
Since $a_0(z)\, h(z)$ and $a_0(z)/h(z)$ are both nonvanishing in 
$\mathbb{C}^+ \setminus \Sigma^+$, the square roots are well-defined 
 and analytic there. 
 By construction, every zero of $a(z)$ in $\C^+$ is a zero of  $a_1(z)$ through the Blaschke factor 
$\mathcal{B}(z)$, while $a_2(z)$ is nonvanishing in $\mathbb{C}^+ \setminus \Sigma^+$.
\begin{remark}
	In the present paper, for simplicity, we assign all zeros of $a(z)$ to the factor $a_1(z)$, so that $a_2(z)$ is analytic and nonvanishing in $\C^+\setminus\Sigma^+$. More generally, one may split the Blaschke product as
	\[
	\mathcal B(z)=\mathcal B_1(z)\mathcal B_2(z),
	\]
	and define
	\[
	a_1(z)=\bigl(a_0(z)h(z)\bigr)^{1/2}\e^{\frac{\i}{2}(x_0^\ell-x_0^r)z}\mathcal B_1(z),\qquad
	a_2(z)=\left(\frac{a_0(z)}{h(z)}\right)^{1/2}\e^{-\frac{\i}{2}(x_0^\ell-x_0^r)z}\mathcal B_2(z),
	\]
	so that $a(z)=a_1(z)a_2(z)$. In this more general factorization, the zeros assigned to $a_1$ produce lower-triangular residue conditions, while the zeros assigned to $a_2$ produce upper-triangular residue conditions and the corresponding poles in $\C^-$ are then determined by Schwarz symmetry. We restrict ourselves to the choice $\mathcal B_2\equiv 1$ in what follows.
\end{remark}

It is straightforward to show the following basic properties of $a_1(z)$ and $a_2(z)$:
\begin{proposition}\label{propa12}
	Let $a_1(z)$ and $a_2(z)$ be as in \eqref{e:defa12}, then we have 
	\begin{enumerate}
		\item $a(z)=a_1(z)a_2(z)$.
		\item
		$a_1(z)=\e^{\mathrm{i}(x_0^\ell-x_0^r)z}(1+\mathcal{O}(z^{-1}))$ and $a_2(z)=1+\mathcal{O}(z^{-1})$,  as $z\to\infty$.
		\item $a_1(z)$ and $a_2(z)$ admit the following relations
		\begin{subequations}\label{e:a12jump}
			\begin{gather}
				\frac{b^*(z_-)\e^{-\mathrm{i} \varphi_0^\ell }}{a_1(z_+)a_2(z_-)} =\frac{b(z_-)\e^{\mathrm{i} \varphi_0^r }}{a_1(z_-)a_2(z_+)}, \quad z\in\Sigma^+.\\
				|a_1(z)|^{-2}=1+\left|\dfrac{b(z)}{a(z)}\right|^2, \quad |a_2(z)|^2=1, \qquad z\in\Real.
			\end{gather}
		\end{subequations}
		\item 
		$a_1(z)$ and $a_2(z)$ are bounded as $z$ approaches any endpoint of $\,\,\Sigma_1\cup\Sigma_2$
		or any intersection point with $\Real$.
	\end{enumerate}
\end{proposition}

Now we define new reflection coefficients $r_1(z)$,  $r_2(z)$ and $\rho(z)$ as follows:
\bse\label{e:r12}
\begin{gather}
	r_1(z)= \frac{ a_2(z_-)}{a_1(z_-)} \frac{\e^{-\i \varphi_0^\ell -2\i x_0^r z}}{a_1(z_+)a_2(z_-)-\i b^*(z_-)\e^{-\i \varphi_0^\ell }}\,, \qquad z\in\Sigma^+,\label{e:defr1}\\
	r_2(z)= \frac{a_1(z_+)}{a_2(z_+)}\frac{\e^{\i \varphi_0^r+2\i x_0^r z}}{a_1(z_+)a_2(z_-)-\i b^*(z_-)\e^{-\i \varphi_0^\ell }}, \qquad z\in\Sigma^+, \label{e:defr2}\\
	\rho(z)=\frac{b(z)}{a_1(z)a_2^*(z)}\e^{-2\i x_0^r z}, \qquad z\in \R. \label{e:rho}
\end{gather}
\ese

Recall that the Sobolev space   $ \mathcal{W}^{n,1}(\Real)$ consists of   functions    $f\in L^1(\R)$,  such that  $\partial_x^j f \in L^1(\mathbb{R})$ for all $0 \le j \le n$ and similarly for  $ \mathcal{W}^{n,1}(\R^\pm)$. Then, following the analysis of \cite{GJZZ}, the reflection coefficient $\rho(z)$
admits the following decay as $z \to \infty$.

\begin{lemma}
   Suppose $u_0-u_e^\ell \in \mathcal{W}^{4,1}(\R^-)$ and $u_0-u_e^r \in \mathcal{W}^{4,1}(\R^+) $. 	
   Then $\rho(z)$ defined in \eqref{e:rho} satisfies $\rho(z)=\mathcal{O}(z^{-4})$ as $z\to\infty$, as follows immediately from Theorem~\ref{theorem1}.
\end{lemma}

To set up a more general Riemann–Hilbert problem, we symmetrize the factorization by defining
\be
M(z;x)=M_o(z;x)
\begin{cases}
	a_2(z)^{\sigma_3}, & z\in \Complex^+,\\
	a_2^*(z)^{-\sigma_3}, & z\in \Complex^-.
\end{cases}
\ee
Then, $M(z;x)$ satisfies the  Riemann--Hilbert problem \eqref{RHP:Mtilde}  with $t=0$.
Following the framework established in \cite{GJZZ} for the elliptic-background setting, the time evolution of the scattering data  introduces the factor $\mathrm{e}^{2\mathrm{i} t z^2}$ into the off-diagonal entries of the jump matrix and the corresponding residue conditions. At the end of this section, we present the theorem regarding the solvability of Riemann--Hilbert problem \ref{RHP:Mtilde} and the existence of solutions for the NLS equation, while their detailed proofs are in Appendix~\ref{a:proof}.

\begin{theorem}
	The potential $u(x,t)$ is reconstructed from the unique solution $M(z;x,t)$ of the Riemann--Hilbert problem RHP~\ref{RHP:Mtilde}, with reflection coefficients  $r_1,r_2\in C(\Sigma)$, and  $\rho\in L^{2,2}(\R)\cap L^{1,2}(\R)$,  and discrete spectral data $\{z_k, \zeta(z_k)\}_{j=1}^N$, via
	\[
	u(x,t) = 2\mathrm{i} \lim_{z \to \infty} z\, M_{12}(z;x,t).
	\]
	Furthermore $u(x,t)\in C^2(\R)\times C(\R)$.
\end{theorem}


\section{Full soliton gas via soliton condensation}\label{sec_soliton} 
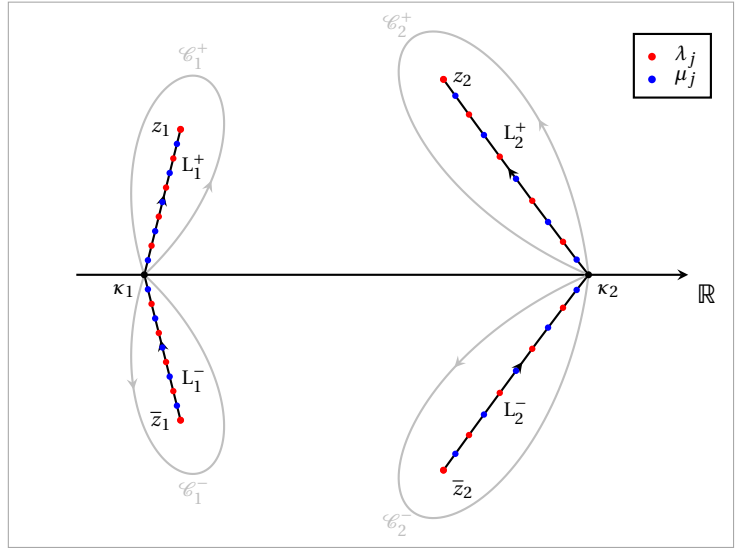
\begin{figure}[t]
	\centering
	\begin{tikzpicture}[scale=0.6, >=stealth, x=1cm,y=1cm]
		\tikzset{
			midarrow/.style={
				postaction={decorate,
				   decoration={markings,
				     mark=at position 0.55 with {\arrow{stealth}}
				   }
				}
			}
		}
		
		\def\N{5} 
		
		\draw[gray!70] (-8,-6) rectangle (8,6);
		
		\coordinate (L) at (-5,0);     
		\coordinate (R) at ( 4.8,0);     
		\coordinate (A) at (-4.2, 3.2);  
		\coordinate (B) at (-4.2,-3.2);  
		\coordinate (C) at ( 1.6, 4.3);  
		\coordinate (D) at ( 1.6,-4.3);  

		\draw[->,thick] (-6.5,0) -- (7.0,0) node[below right] {$\mathbb{R}$};
		
		\draw[thick,midarrow] (L) -- (A);
		\draw[thick,midarrow] (B) -- (L);
		
		\draw[thick,midarrow]
		(R) .. controls (4.0,1.0) and (2.4,3.2) .. (C);
		\draw[thick,midarrow]
		(D) .. controls (2.4,-3.2) and (4.0,-1.0) .. (R);
		
		\draw[thick, gray!50, domain= -0.5236:0.5236, variable=\x, samples=100,->-=0.25] 
			plot ( { -5 + 5*cos(deg(\x+1.33))*cos(deg(3*\x)) },  { 4.5*sin(deg(\x+1.33))*cos(deg(3*\x))} );
		\draw[thick,gray!50, domain=-pi/6:pi/6, variable=\x, samples=100,->-=0.25] 
			plot ( { -5 + 5*cos(deg(\x-1.33))*cos(deg(3*\x)) },  { 4.5*sin(deg(\x-1.33))*cos(deg(3*\x))} );
		\draw[thick, gray!50, domain= -0.5236:0.5236, variable=\x, samples=100,->-=0.25] 
			plot ( { 4.8 + 6.5*cos(deg(\x+2.21))*cos(deg(3*\x)) },  { 6.5*sin(deg(\x+2.21))*cos(deg(3*\x))} );
		\draw[thick,gray!50, domain=-pi/6:pi/6, variable=\x, samples=100,->-=0.25] 
			plot ( { 4.8 + 6.5*cos(deg(\x-2.21))*cos(deg(3*\x)) },  { 6.5*sin(deg(\x-2.21))*cos(deg(3*\x))} );

		\fill (L) circle (2.2pt);
		\fill (R) circle (2.2pt);
		
		\fill[red] (A) circle (2.2pt);
		\fill[red] (B) circle (2.2pt);
		\fill[red] (C) circle (2.2pt);
		\fill[red] (D) circle (2.2pt);
		
		\node[below left] at (L) {\scriptsize$\kappa_1$};
		\node[below right] at (R) {\scriptsize$\kappa_2$};
		\node[left]  at (A) {\scriptsize$z_1$};
		\node[left]  at (B) {\scriptsize$\overline{z}_1$};
		\node[right] at (C) {\scriptsize$z_2$};
		\node[below right] at (D) {\scriptsize$\overline{z}_2$};
		\node[gray!50] at (-3.9,4.8)  {\scriptsize $\mathcal{C}_1^+$};
		\node[gray!50] at (-3.9,-4.8)  {\scriptsize $\mathcal{C}_1^-$};
		\node[gray!50] at (0.6,5.5)  {\scriptsize $\mathcal{C}_2^+$};
		\node[gray!50] at (0.6,-5.5)  {\scriptsize $\mathcal{C}_2^-$};
		
		\node[above] at (-3.9,1.9) {\scriptsize$\L_1^+$};
		\node[below] at (-3.9,-1.9) {\scriptsize$\L_1^-$};
		\node[above] at (3.2,2.6) {\scriptsize$\L_2^+$};
		\node[below] at (3.2,-2.6) {\scriptsize$\L_2^-$};
		
		\foreach \l in {1,...,\N} {
			\pgfmathsetmacro{\s}{(\l-0.5)/\N}
			\fill[blue] ($(L)!\s!(A)$) circle (2pt);
		}
		\foreach \l in {1,...,\N} {
			\pgfmathsetmacro{\s}{\l/\N}
			\fill[red] ($(L)!\s!(A)$) circle (2pt);
		}
		
		\foreach \l in {1,...,\N} {
			\pgfmathsetmacro{\s}{(\l-0.5)/\N}
			\fill[blue] ($(L)!\s!(B)$) circle (2pt);
		}
		\foreach \l in {1,...,\N} {
			\pgfmathsetmacro{\s}{\l/\N}
			\fill[red] ($(L)!\s!(B)$) circle (2pt);
		}
		
		\foreach \l in {1,...,\N} {
			\pgfmathsetmacro{\t}{(\l-0.5)/\N}
			\path (R) .. controls (4.0,1.0) and (2.4,3.2) .. (C)
			coordinate[pos=\t] (P);
			\fill[blue] (P) circle (2pt);
		}
		\foreach \l in {1,...,\N} {
			\pgfmathsetmacro{\t}{\l/\N}
			\path (R) .. controls (4.0,1.0) and (2.4,3.2) .. (C)
			coordinate[pos=\t] (Q);
			\fill[red] (Q) circle (2pt);
		}
		
		\foreach \l in {1,...,\N} {
			\pgfmathsetmacro{\t}{(\l-0.5)/\N}
			\path (R) .. controls (4.0,-1.0) and (2.4,-3.2) .. (D)
			coordinate[pos=\t] (P);
			\fill[blue] (P) circle (2pt);
		}
		\foreach \l in {1,...,\N} {
			\pgfmathsetmacro{\t}{\l/\N}
			\path (R) .. controls (4.0,-1.0) and (2.4,-3.2) .. (D)
			coordinate[pos=\t] (Q);
			\fill[red] (Q) circle (2pt);
		}
		
		\begin{scope}[shift={(1,1.5)}]
		\draw[thick] (4.8,2.4) rectangle (6.5,3.8);
		\fill[red]  (5.2,3.3) circle (2.4pt);
		\node[right] at (5.5,3.3) {\scriptsize$\lambda_j$};
		\fill[blue] (5.2,2.8) circle (2.4pt);
		\node[right] at (5.5,2.8) {\scriptsize$\mu_j$};
		\end{scope}
	\end{tikzpicture}
	\caption{Distribution of solitons poles along the accumulation curves $L_1$ and $L_2$ and the interpolation contours $\mathcal{C}_1$ and $\mathcal{C}_2$
	used to remove the poles. 
	}
	\label{fig:solgas}
\end{figure}


\subsection{Generalized gas as limit of $4N$ soliton}

In this section we describe the construction of a soliton gas as the continuum limit of a multi-soliton solution of \eqref{e:nls} in which the discrete eigenvalues of the Zakharov-Shabat scattering problem \eqref{e:zs} associated to the soliton solution accumulate along given curves. We show below the continuum limit soliton gas is closely related to the step-like oscillatory potentials we have introduced above. 

We start by considering a $4N$-soliton solution of the focusing NLS equation \eqref{e:nls} spectrally encoded by $8N$ discrete eigenvalues of \eqref{e:zs} together with their associated nonzero norming constants. By Schwarz symmetry these $8N$ eigenvalues come in complex conjugate pairs. 
We suppose the eigenvalues are arranged on two non-intersecting line segments $\L_1^+$ and $\L_2^+$ which emerge from distinct points $\kappa_1, \kappa_2 \in \R$ into $\mathbb{C}^+$ which terminate at endpoints $z_1$ and $z_2$, respectively, together with their complex conjugates which we denote by $\L_1^-$ and $\L_2^-$. 
See Figure~\ref{fig:solgas}. 
For simplicity of presentation we suppose that $2N$ eigenvalues lie on each line $\L_\alpha^+$, $\alpha \in \{1,2\}$ (with the other $4N$ eigenvalues at the complex conjugate points) and that these eigenvalues are equally spaced along each line. We partition the eigenvalues along each line into two interlaced sets $\{ \lambda^{(\alpha)}_j\}_{j=1}^N, \ \{ \mu^{(\alpha)}_j\}_{j=1}^N \subset \L_\alpha^+$. The $4N$-soliton solution of \eqref{e:nls} can then be described by a meromorphic Riemann-Hilbert problem with simple poles at each discrete eigenvalue as follows.


\begin{RHP}
	\label{rhp:2N}
	Find a $2 \times 2$ matrix valued function $\Msol(\cdot; x, t)$ with the properties:
	\begin{enumerate}
		\item $\Msol(z; x, t)$ is analytic for $z \in\C\setminus \left\{ \lambda_j^{(1)}, \lambda_j^{(2)}, \overline{\lambda_j^{(1)}},  \overline{\lambda_j^{(2)}}, \mu_j^{(1)}, \mu_j^{(2)} , \overline{\mu_j^{(1)}}, \overline{\mu_j^{(2)} } \right\}_{j=1}^N $\,.
		\item$\Msol(z; x, t) = I + \mathcal{O}(z^{-1})$
		as $z \to \infty$\,.
		\item $\Msol(z; x,t)$ has simple poles at $\{\lambda_j^{(\alpha)},\overline{\lambda_j^{(\alpha)}}\}_{j=1}^N$\,,
		$\{\mu_l^{(\alpha)}, \overline{\mu_l^{(\alpha)}}\}_{l=1}^M$, with $\alpha\in\{1,2\}$:
		\bse\label{e:sol.residues}
		\begin{gather}
		\hspace{-.5cm}
		\Res_{z=\lambda_j^{(\alpha)}} \Msol(z)  =\lim_{z\to \lambda_j^{(\alpha)}} \Msol(z) 
		  \begin{psmallmatrix} 0&0\\ c_j^{(\alpha)}\e^{2\theta(z;x,t)} & 0\end{psmallmatrix} \,,\quad 
		\Res_{z=\overline{\lambda_j^{(\alpha)}}} \Msol(z) =  \lim_{z\to \overline{\lambda_j^{(\alpha)} } } \Msol(z) 
		  \begin{psmallmatrix} 0&-\overline{c_j^{(\alpha)}}\e^{-2\theta(z;x,t)}\\ 0 &0\end{psmallmatrix}\, ,
		\\
		\hspace{-.5cm}		
		\Res_{z=\mu_j^{(\alpha)}} \Msol(z) =\lim_{z\to \mu_j^{(\alpha)}} \Msol (z) 
		  \begin{psmallmatrix} 0&\chi_j^{(\alpha)} \e^{-2\theta(z;x,t)}\\ 0 &0\end{psmallmatrix}\,, \quad 
		\Res_{z=\overline{\mu_j^{(\alpha)}}} \Msol(z) = \lim_{z\to \overline{\mu_j^{(\alpha)}}} \Msol(z) 
		\begin{psmallmatrix} 0&0\\ -\overline{\chi_j^{(\alpha)}}\e^{2\theta(z;x,t)} &0\end{psmallmatrix}\,,
		\end{gather}
		\ese
			where 
		\[
			\theta(z;x,t) = \mathrm{i}t z^2 + \mathrm{i}x z\,.
		\]
	\end{enumerate}
\end{RHP}

\begin{remark}
RHP~\ref{rhp:2N} differs from the standard representation of the $N$-soliton Riemann-Hilbert problem in that here we have  upper- and lower-triangular residue conditions. These comes from a suitable renormalization of the standard meromorphic Riemann--Hilbert problem associated to the $N$-soliton solution of \eqref{e:nls} which have  residues of one triangularity, and the detailed renormalization can be found in Appendix~B.1 of \cite{BJM18}.
\end{remark}

\subsection{From a $\mathcal{C}$-jump RHP to a pure pole RHP}
To arrive at a soliton gas from the $4N$ soliton we need to assume some coherence in the eigenvalues and the associated normalizing coefficients $c^{(\alpha)}_j$ and $\chi^{(\alpha)}_j$ which appear in RHP~\ref{rhp:2N}. To make the connection with minimal technical assumptions we assume that the two poles sets are equally spaced along each $\L_\alpha^+$ (and the same in the lower half plane) and interlacing:
\bse \label{e:interlacing}
\begin{gather}
	\lambda^{(1)}_ j=\kappa_1+j\frac{z_1-\kappa_1}{N}\,,\qquad
	\mu_j^{(1)}=\kappa_1+\left(j-\frac{1}{2}\right)\frac{z_1-\kappa_1}{N}\,,
	\qquad j=1,\dots, N, \\
	\lambda_j^{(2)}=\kappa_2+j\frac{z_2-\kappa_2}{N}\,, 
	\qquad 
	\mu_j^{(2)}=\kappa_2+\left(j-\frac{1}{2}\right)\frac{z_2-\kappa_2}{N}\,,
	\qquad j=1,\dots, N.
\end{gather}
\ese
Next we describe the conditions we impose upon the norming constants. Let 
\be\label{e:defB}
B(w)=\prod_{\ell=1}^N\left(\frac{w-\ell+\frac{1}{2}}{w-\ell}\right)=\frac{\Gamma(w-N)}{\Gamma(w)}\frac{\Gamma(w+\frac{1}{2})}{\Gamma(w+\frac{1}{2}-N)}.
\ee
Also introduce the simple shift and rescaling transformations
\be \label{e:defw}
w_1(z)=\frac{z-\kappa_1}{z_1-\kappa_1}, \qquad
w_2(z)=\frac{z-\kappa_2}{z_2-\kappa_2}.
\ee
Then the function
\begin{equation}
	B( N w_\alpha(z)), \qquad \alpha \in \{1,2\}
\end{equation}
has poles at $\{\lambda_\ell^{(\alpha)} \}_{\ell=1}^N$ and zeros at $\{ \mu_\ell^{(\alpha)} \}_{\ell=1}^N$ for each $\alpha \in \{ 1,2 \}$. 
We now make the following assumption:
\begin{assumption}\label{assume}
There exists analytic functions $C(z)$ and $D(z)$ defined in a neighborhood of~~$\Sigma_1^+ \cup \Sigma_2^+$ in $\C^+$ which satisfy
\be\label{e:vanishing}
	|C(z)| = \bigo{z-\kappa_\alpha},  \qquad |D(z)| = \bigo{z- \kappa_\alpha}, \qquad k \to \kappa_\alpha, \quad \alpha \in \{1,2\}.
\ee
and 
\be
	| 1+C(z) D(z) | > \delta, 
\ee
for some constant $\delta > 0$ and all $z$ in a neighborhood of 	~$\Sigma_1^+ \cup \Sigma_2^+$ in $\C^+$.
Moreover,  the norming constants $c_j^{(\alpha)}$ and $\chi^{(\alpha)}_j$ in \eqref{e:sol.residues}, associated to the poles $\lambda_j^{(\alpha)},\ \mu_j^{(\alpha)}$ defined by \eqref{e:interlacing}, are given by
\be\label{e:normingconst}
	c_j^\alpha= -C(\lambda_j^{(\alpha)}) \left( \frac{z_\alpha-\kappa_\alpha}{N}\right) \Res_{w=j} B(w) ,\qquad
	\chi_j^\alpha= -\frac{D(\mu_j^{(\alpha)})}{1+C(\mu_j^{(\alpha)})D(\mu_j^{(\alpha)})} \left(\frac{ z_\alpha-\kappa_\alpha}{N}\right) B'(j-1/2)^{-1}.
	\ee
\end{assumption}

This assumption allows us to introduce a sectionally holomorphic Riemann-Hilbert problem which is equivalent to the meromorphic problem RHP~\ref{rhp:2N} which trades the poles for jumps on loop contours which enclose the pole loci $\L_1$ and $\L_2$. 
For $\alpha \in \{1,2\}$ introduce simple closed contours $\mathcal{C}_\alpha^+$ in $\C^+$ which enclose $\L_\alpha^+$, meeting the real axis at $\kappa_\alpha$, and stay within the domain of analyticity of the functions $C(z)$ and $D(z)$ described in Assumption~\ref{assume}, see Figure~\ref{fig:solgas}.
Let $\mathcal{C}_\alpha^-$ denote their complex conjugate contours in $\C^-$ and set
\[
	\mathcal{C}_1 := \mathcal{C}_1^+ \cup \mathcal{C}_1^-,
	\qquad
	\mathcal{C}_2 := \mathcal{C}_2^+ \cup \mathcal{C}_2^-.
\]

Now we define a new Riemann-Hilbert problem:
\begin{RHP}\label{rhp:P}
	Given contours $\mathcal{C}_1$ and $\mathcal{C}_2$ as shown in Figure~\ref{fig:solgas} and functions $C(z)$ and $D(z)$ satisfying Assumption~\ref{assume},
	find a matrix  $P (z) = P(z; x, t)$ such that:
	\begin{enumerate}
		\item $P (z)$ is analytic on $\C \setminus (\mathcal{C}_1 \cup \mathcal{C}_2)$.
		\item $P(z)$ has jumps across $\mathcal{C}_1 \cup \mathcal{C}_2$ with the jump condition $P_+(z)=P_-(z) V_P(z)$, where the jump matrix is  given as follows 
		\be\label{e:Pjump}
		\begin{gathered}
		V_P(z) = \begin{dcases}
			\begin{pmatrix}
				1 & D(z)B(Nw_\alpha(z))^{-1}\e^{-2\theta(z;x,t)}\\
				C(z)B(Nw_\alpha(z))\e^{2\theta(z;x,t)} & 1+C(z)D(z)
			\end{pmatrix}\,, & z \in \mathcal{C}_\alpha^+\, ,\\
			\begin{pmatrix}
				1+C^*(z)D^*(z) & -C^*(z)B(Nw_\alpha^*(z))\e^{-2\theta(z;x,t)}\\ -D^*(z)B(Nw_\alpha^*(z))^{-1}\e^{2\theta(z;x,t)}
				& 1
			\end{pmatrix}\,, & z \in \mathcal{C}_\alpha^-,
		\end{dcases}
		\end{gathered}
		\ee
		for $\alpha \in \{1,2\}$.
		\item $P(z) = I + \mathcal{O}(z^{-1})$
		as $z \to \infty$\,.
	\end{enumerate}
\end{RHP}

\begin{remark}
	The vanishing condition \eqref{e:vanishing} in Assumption~\ref{assume} implies that $\lim_{z \to \kappa_\alpha} V_P(z) = I$ along any component of $\mathcal{C}_\alpha$, so the jump matrix is trivially consistent at the self-intersection points of the jump contour. 
\end{remark}

\begin{proposition}
	The solution to RHP~\ref{rhp:P} exists and is identical to the solution of $\Msol(z)$ of RHP~\ref{rhp:2N} with poles given by \eqref{e:interlacing} and norming constants  given by \eqref{e:normingconst} for all $z$ outside a compact set. In particular, the $4N$ soliton solution of \eqref{e:nls} encoded by RHP~\ref{rhp:2N} is given by 
\[
	u_{\mathrm{sol}}(x,t) = 2\i \lim_{z \to \infty} z P_{12}(z;x,t)
\]
\end{proposition}

\begin{proof}
	A unique solution of RHP~\ref{rhp:2N} exists for any choice of norming constants provided all of the poles are distinct. See \cite[Appendix B]{BJM18} for details. 
	Supposing that the poles and norming constants of RHP~\ref{rhp:2N} are given by \eqref{e:interlacing} and \eqref{e:normingconst} respectively, define 
	\be\label{e:defP}
	P(z)=\begin{dcases}
	\Msol(z)\,
	\begin{psmallmatrix}
        	1 & 0\\
        	C(z)B(Nw_\alpha(z))\e^{2\theta(z;x,t)} & 1
        \end{psmallmatrix}
	\begin{psmallmatrix}
	1 & D(z)B(Nw_\alpha(z))^{-1}\e^{-2\theta(z;x,t)}\\
	0 & 1
	\end{psmallmatrix}, & z \in \mbox{Int}(\mathcal{C}^+_\alpha)\\
	\Msol(z)
	\begin{psmallmatrix}
	1 & 	-C^*(z)\,B\left(Nw^*_\alpha(z)\right)\,\e^{-2\theta( z;x,t)}
	\\[2pt]
	0 & 1
	\end{psmallmatrix}
	\begin{psmallmatrix}
	1 & 0\\[2pt]
	-D^*(z)\,B\left(Nw^*_\alpha(z)\right)^{-1}\,\e^{2\theta( z;x,t)}& 1
	\end{psmallmatrix},
	& z\in \mathrm{Int}(\mathcal{C}^-_\alpha),\\
	\Msol(z),& \mbox{otherwise}.
	\end{dcases}
	\ee
	It's a simple computation to show that $P(z)$ is meromorphic in $\C \setminus (\mathcal{C}_1 \cup \mathcal{C}_2)$ and satisfies the jump relation \eqref{e:Pjump}. The 	normalization condition at infinity follows immediately from the equivalent property for $\Msol$ since the transformation \eqref{e:defP} is an identity outside a compact set. It only remains to show that the transformation defining $P$ removes the poles. First consider a point $\lambda_j^{(\alpha)}$, $j \in \{1,\dots, N\}$. At these points $B(N w_\alpha(z))$ has a simple pole and 
	\[
		\Res_{z=\lambda^{(\alpha)}_j} B( N w_\alpha(z)) = \frac{1}{N w'_\alpha(\lambda^{(\alpha)}_j)} \Res_{w=j} B(w) 
		= \frac{z_\alpha - \kappa_\alpha}{N} \Res_{w=j} B(w).
	\]
	Since the second column of $\Msol(z)$ is analytic at $\lambda_j^{(\alpha)}$, it follows from \eqref{e:defP} that $P$ has at most simple poles at each $\lambda_j^{(\alpha)}$. 
	Using \eqref{e:sol.residues} and \eqref{e:normingconst} we have
	\begin{align*}
    		\Res_{z=\lambda^{(\alpha)}_j} P(z) 
    		&= \lim_{z \to \lambda^{(\alpha)}_j} 
    		\Msol(z) \begin{psmallmatrix} z-\lambda_j^{(\alpha)} & 0 \\ 0 & 1 \end{psmallmatrix} \cdot
		\begin{psmallmatrix}
            		1 & 0\\
            		(z-\lambda_j^{(\alpha)})C(z)B(Nw_\alpha(z)) \e^{2\theta(z;x,t)} & z-\lambda_j^{(\alpha)}
            	\end{psmallmatrix}
    		\begin{psmallmatrix}
    			1 &  D(z)B(Nw_\alpha(z))^{-1}\e^{-2\theta(z;x,t)}\\
    			0 & z-\lambda_j^{(\alpha)}
    		\end{psmallmatrix} \\
		&= \Msol_2((\lambda_j^{(\alpha)})) 
		\begin{bmatrix} c_j^{(\alpha)} \e^{2\theta(\lambda_j^{(\alpha)};x,t)}  & 1\end{bmatrix} 
		\begin{psmallmatrix} 1 & 0 \\ C(\lambda_j^{(\alpha)}) \Res_{z=\lambda^{(\alpha)}_j} B( N w_\alpha(z)) \e^{2\theta(\lambda_j^{(\alpha)};x,t)}   & 0 \end{psmallmatrix} 
		= \begin{psmallmatrix} 0 & 0 \\ 0 & 0 \end{psmallmatrix},
	\end{align*}
	where $\Msol_2(z)$ denotes the second column of the matrix $\Msol(z)$. So $P$ is regular at each $\lambda^{(\alpha)}_j$. 
	Now consider a point $\mu_j^{(\alpha)}$, $j \in \{1,\dots, N\}$. At these points $B(N w_\alpha(z))$ has a simple zero and 
	\[
		\Res_{z=\mu^{(\alpha)}_j} B( N w_\alpha(z))^{-1} = \frac{1}{N w'_\alpha(\mu^{(\alpha)}_j)} \Res_{w=j-1/2} B(w)^{-1} 
		= \frac{z_\alpha - \kappa_\alpha}{N} \Res_{w=j-1/2} B(w)^{-1}.
	\]
	Since the first column of $\Msol(z)$ is regular, the second has a simple pole, and $B$ has a simple zero at $\mu_j^{(\alpha)}$, \eqref{e:defP} implies $P$ has at most a simple pole at $\mu_j^{(\alpha)}$. Computing the residue using \eqref{e:sol.residues} and \eqref{e:normingconst} gives
	\begin{align*}
    		\Res_{z=\mu^{(\alpha)}_j} P(z) 
    		&= \lim_{z \to \mu^{(\alpha)}_j} 
    		\Msol(z) 
		\begin{psmallmatrix}
            		1 & 0\\
            		C(z)B(Nw_\alpha(z)) \e^{2\theta(z;x,t)} & 1
            	\end{psmallmatrix}
    		\begin{psmallmatrix} 1&0 \\ 0 & z-\mu_j^{(\alpha)}  \end{psmallmatrix} \cdot
		\begin{psmallmatrix}
    			z-\mu_j^{(\alpha)} &  (z-\mu_j^{(\alpha)}) D(z)B(Nw_\alpha(z))^{-1}\e^{-2\theta(z;x,t)}\\
    			0 & 1
    		\end{psmallmatrix} \\
		&= \Msol_1((\mu_j^{(\alpha)})) 
		\begin{bmatrix} 1+ \chi_j^{(\alpha)} C(\mu_j^{(\alpha)}) \frac{N B'(j+\frac{1}{2})}{z_\alpha - \kappa_\alpha}  & \chi_j^{(\alpha)} \e^{-2\theta(\mu_j^{(\alpha)};x,t)}  \end{bmatrix} 
		\begin{psmallmatrix} 0 &  \frac{D(\mu_j^{(\alpha)})}{B'( j-1/2)} \frac{z_\alpha - \kappa_\alpha}{N}  \e^{-2\theta(\mu_j^{(\alpha)};x,t)} \\  0 & 1 \end{psmallmatrix} 
		 = \begin{psmallmatrix} 0 & 0 \\ 0 & 0 \end{psmallmatrix}.
	\end{align*}	
	Which shows that poles at $z = \mu_j^{(\alpha)}$ are removed as well. Similar computations can be done to show that $P(z)$ is pole free in $\C^-$ as well. The computation is as above, up to the Schwarz conjugation symmetry. The details are omitted.  
	It follows that the matrix $P(z)$ defined by \eqref{e:defP} is a solution of RHP~\ref{rhp:P}. 
	\end{proof}

\subsection{Large-$N$ limit and the full soliton-gas RHP}
To pass to the continuum limit we need to understand the asymptotic behavior of the term $B(N w_\alpha(z))$ appearing in \eqref{e:Pjump} in the limit as $N \to \infty$. 
The following proposition describes the limiting behavior of $B(N w_\alpha(z))$  in two regimes: (1) on sets bounded away from the curve $\L_\alpha$ on which it accumulates zeros and poles; and (2) in a neighborhood of  the point $\kappa_\alpha$, where the zeros approach the real axis, but outside a sector centered on $\L_\alpha$.
To state our theorem, we define
\be\label{e:sector}
		S_\alpha = S_\alpha(\phi) = \left\{ z \in \C \,:\, \left| \arg\left( \frac{z - \kappa_\alpha}{z-z_\alpha} \right) \right| \in (\phi, \pi]  \right\},  \qquad \phi \in (0,\pi], \quad \alpha \in \{1,2\},
\ee
the sectors of central angle $2 \phi$ centered on the lines $\L_\alpha$. 
\begin{proposition}\label{prop:asympB}
	Fix $\alpha\in\{1,2\}$ and let $B(w)$ and $w_\alpha(z)$ be defined by \eqref{e:defB} and \eqref{e:defw}.  
Define
	\begin{equation}
	\label{Phi}
	\Phi_\alpha(z):=\left(\frac{z-\kappa_\alpha}{z-z_\alpha}\right)^{1/2},
	\end{equation}
	where the square root is chosen analytic in $\mathbb C\setminus \Sigma_\alpha^+$ and normalized by
	$\Phi_\alpha(z)\to 1$ as $z\to\infty$
	\paragraph{(i) }
	For any compact set $K \subset \C^+ \setminus \L_\alpha^+$,
	\be\label{e:asympB1}
	B(N w_\alpha(z))=\Phi_\alpha(z)\left(1+\bigo{\frac{1}{N}} \right),
	\qquad z\in K, \quad N \to \infty
	\ee
	uniformly for $z\in K$.

	\paragraph{(ii) } Let $F(z)$ satisfy $|F(z)| \leq A |z - \kappa_\alpha|$ in an open neighborhood of $z = \kappa_\alpha$. Then
	\be\label{e:asympB2}
		\left| F(z) B(N w_\alpha(z)) - F(z) \Phi_\alpha(z) \right| = \bigo{\frac{1}{N}}, 
		 \qquad z \in S_\alpha(\phi), \quad N \to \infty.
	\ee
The convergence is uniform for $\phi > \phi_0 > 0$. 
\end{proposition}

\begin{proof}
	The result relies on the following property of Gamma functions \cite[\href{https://dlmf.nist.gov/5.11.13}{(5.11.13)}]{dlmf};  for fixed $a ,b$ 
	\be\label{e:stirling}
	\frac{\Gamma(z+a)}{\Gamma(z+b)}
		= z^{a-b} \left[ 1+ \frac{(a-b)(a+b-1)}{2z} + \bigo{\frac{1}{z^2}} \right], \qquad z \to \infty, \quad |\arg(z)| < \pi - \epsilon
	\ee
	Applying this result directly to \eqref{e:defB} one has that 
	\begin{subequations}\label{e:Basymp}
	\be
		B(w) = \sqrt{\frac{w}{w-N}} \left[1 + \bigo{\frac{1}{w}} + \bigo{\frac{1}{w-N}} \right],  \qquad w \to \infty, \quad | \arg{w-N} | < \pi - \epsilon,
	\ee
	To get an expansion of $B(w)$ for large $w$ in the missing sector we make use of the reflection formula $\Gamma(z) \Gamma(1-z) = \frac{\pi}{\sin(\pi z)}$ to write
	\[
		B(w) = \frac{\Gamma(1-w)} {\Gamma(\tfrac{1}{2} -w)}\frac{\Gamma(N + \frac{1}{2} - w)}{\Gamma(1+N-w)} .
	\]
	Applying \eqref{e:stirling} to this representation for $B(w)$ gives
	\be\label{e:blaschke2}
		B(w) = \sqrt{\frac{-w}{N-w}}\left[ 1+ \bigo{\frac{1}{-w}} + \bigo{\frac{1}{N-w}}\right],  \qquad w \to \infty, \quad |\arg(w)| >  \epsilon. 
	\ee
	\end{subequations}
	Recognizing that the map $w_\alpha(z)$, $\alpha \in \{1,2\}$, defined by \eqref{e:defw} maps $\L^+_\alpha$ to $[0,1]$, any point $z$ in a compact subset $K$ of $\C \setminus \L^+_\alpha$ satisfies $\min\limits_{z\in K, w \in [0,N]} |Nw_\alpha(z) - w | > cN $ for some fixed $c = c(K) > 0$. Combining this with the overlapping asymptotics in \eqref{e:Basymp} gives \eqref{e:asympB1} upon observing that $\sqrt{\frac{w_\alpha(z)}{w_\alpha(z)-1}} = \Phi_\alpha(z)$. 
	
	The function $B(N w_\alpha(z))$ is not uniformly close to $\Phi_\alpha(z)$ for $z$ near $\kappa_\alpha$, the real endpoint of $\L_\alpha^+$. This is because $w_\alpha( \kappa_\alpha)  = 0$, so there is a small $O(N^{-1})$ disk centered at $z = \kappa_\alpha$ where we cannot uniformly approximate the $\Gamma(1-N w_\alpha(z))/\Gamma(-N w_\alpha(z))$ in \eqref{e:blaschke2} using \eqref{e:stirling}. Fix a small, $N$-independent, neighborhood $U$ of $\kappa_\alpha$ bounded away from $z_\alpha$, and take $S_\alpha(\phi)$ as in \eqref{e:sector}. Then starting from \eqref{e:blaschke2} 
\[
	B(w) = \frac{1}{\sqrt{N-w}}\frac{\Gamma(1-w)}{\Gamma(\frac{1}{2}-w)} \left[ 1 + \bigo{\frac{1}{N-w}} \right], \qquad N \to \infty, \quad  |\arg(w)| > \epsilon.
\]
It follows that 
\[
	B(N w_\alpha(z)) = \Phi_\alpha(z) \frac{\Gamma(1-N w_\alpha(z))}{\sqrt{-N w_\alpha(z)} \Gamma(\frac{1}{2} -N w_\alpha(z))} \left[ 1 + \bigo{\frac{1}{N}} \right],  \qquad z \in U.
\]
Now suppose that $F(z)$ is any function satisfying the linear bound $|F(z)| < A |z - \kappa_\alpha|$ for $z \in U$. Then 
\[
	\left| F(z)  B(Nw_\alpha(z)) - F(z) \Phi_\alpha(z)   \right| \leq  A | z_\alpha - \kappa_\alpha |   \left| \Phi_\alpha(z) G(N w_\alpha(z)) \right| \left[ 1+ \bigo{\frac{1}{N}} \right], 
	\qquad  N \to \infty, \quad  |\arg(w)| > \epsilon,
\]
where
\[
	G(w) := w\left[ 1 - \frac{\Gamma(1-w) }{\sqrt{-w}  \Gamma(\tfrac{1}{2} - w)} \right].
\]	
This function is analytic for $w \in \C \setminus [0,\infty)$, and satisfies both $G(0) = 0$ and, again using \eqref{e:stirling}, 
\[
	G(w) = \frac{1}{8} + \bigo{w^{-1}}, \quad w \to \infty, \qquad |\arg(w)| > \epsilon. 
\]
This shows that $G(w)$ is uniformly bounded in any closed sector $| \arg w | > \phi$ for fixed $\phi \in (0,\pi)$. 
Since the function $w=w_\alpha(z)$ maps $S_\alpha(\phi)$ onto the region $|\arg w| > \phi$, it follows that $|G(N w_\alpha(z))|$ is bounded independent of $N$ for $z \in S_\alpha(\phi)$. 
The bound \eqref{e:asympB2} follows immediately. 
\end{proof}

Proposition \ref{prop:asympB}, combined with Assumption~\ref{assume}, justify replacing the function $B(N w_\alpha(z))$ with $\Phi_\alpha(z)$, $\alpha \in \{1,2\}$ in the large $N$ limit. This defines a continuum Riemann-Hilbert problem.

\medskip
\noindent
\textbf{RHP 4.3 (Continuum limit problem).}\label{rhp:Pi}
Find a $2\times 2$ matrix-valued function $P_\infty(z)=P_\infty(z; x,t)$ such that
\begin{enumerate}
	\item $P_\infty(z)$ is analytic in $\C\setminus (\mathcal{C}_1\cup\mathcal{C}_2)$.
	
	\item $P_\infty(z)=I+O(z^{-1})$ as $z\to\infty$.
	
	\item For $z\in \mathcal{C}_1 \cup \mathcal{C}_2$, $P_\infty$ has continuous boundary values which satisfy
	\be
	\begin{gathered}
	P_{\infty,+}(z)=P_{\infty,-}(z) V_{P_\infty}(z;x,t) , \\
	V_{P_\infty}(z;x,t) = \begin{dcases}
	\begin{pmatrix}
		1 & D(z)\Phi_\alpha(z)^{-1}\e^{-2\theta(z; x,t)}\\
		C(z)\Phi_\alpha(z)\e^{2\theta(z; x,t)} & 1+C(z)D(z)
	\end{pmatrix}, 
	& z\in \mathcal{C}_\alpha^+, \quad  \alpha=1,2,\\
	\begin{pmatrix}
		1+C^*(z)D^*(z) & -C^*(z)\Phi_\alpha^*(z)^{-1}\e^{-2\theta( z; x,t)}\\
		-D^*(z)\Phi_\alpha^*(z)\e^{2\theta( z; x,t)} & 1
	\end{pmatrix}, 
	& z\in \mathcal{C}_\alpha^-, \quad \alpha =1,2.
	\end{dcases}
	\end{gathered}
	\ee
\end{enumerate}

\begin{proposition}\label{bounded}
	A unique solution of RHP~\ref{rhp:Pi} exists for choice of functions $C(z)$ and $D(z)$ which satisfy Assumption~\ref{assume}. 
	The solution $P_\infty$ assumes continuous (and bounded)  boundary values on $\mathcal{C}$. 
\end{proposition}

\begin{proof}
	This follows from a very standard argument, we sketch the proof in broad strokes omitting the details. 
	The uniqueness of solutions, provided they exist, follows from a standard Liouville argument applied to the ratio of any two solutions. 
	The existence of solutions follows from the application of Zhou's vanishing lemma \cite[Theorem~9.3]{Zhou}. We remark that \eqref{e:vanishing} guarantees that $V_{P_\infty}(\kappa_\alpha) = I$, $\alpha=1,2$, which ensures the jumps are locally self-consistent at the points of self-intersection. Zhou's vanishing lemma only establishes existence of a solution with $L^2$ boundary values. However, the analyticity condition in Assumption~\ref{assume} implies the contours are locally deformable. By considering the solution to the problem with deformed jumps, it follows the original solution takes continuous boundary values on $\mathcal{C}$ which are everywhere bounded.
\end{proof}

\begin{proposition}\label{prop:gas.error}
Let $P(z)$ be the unique solution to RHP~\ref{rhp:P}, and let $P_\infty(z)$ be the unique solution to RHP~\ref{rhp:Pi}. Define
\be\label{e:error}
\E(z)=P(z)P_\infty(z)^{-1}.
\ee
Then, for $N$ sufficiently large, $\E(z)$ satisfies a normalized Riemann--Hilbert problem on
$\mathcal{C}_1\cup \mathcal{C}_2$, is analytic in $\C\setminus(\mathcal{C}_1\cup \mathcal{C}_2)$, and
\[
\E(z)=I+O(z^{-1}),\qquad z\to\infty .
\]
Moreover,
\[
\E(z)=I+O(N^{-1/2}),\qquad N\to\infty,
\]
uniformly for $z$ in any compact subset of $\mathbb C\setminus(\mathcal{C}_1\cup \mathcal{C}_2)$.

\end{proposition}

\noindent The proof is given in Appendix~\ref{app:solgas.error}. 

The final step in our analysis of the soliton gas problem is to collapse the artificially introduced interpolation contours $\mathcal{C}$ by collapsing the jumps of the continuum problem back onto the contours $\L_\alpha$ on which the solitons have been ``condensed" in the soliton gas limit. 
Define $W_\infty(z;x,t)$ 
\be\label{e:defWinf}
W_\infty(z;x,t)=
\begin{cases}
	P_\infty(z;x,t)
	\begin{pmatrix}
		1+C(z)D(z) & -D(z)\Phi_\alpha(z)^{-1}\e^{-2\theta(z; x,t)}\\
		-C(z)\Phi_\alpha(z)\e^{2\theta(z; x,t)} & 1
	\end{pmatrix}, & z\in \mbox{Int}(\mathcal{C}_\alpha^+),\ \alpha=1,2,\\[2.2ex]
	P_\infty(z)
	\begin{pmatrix}
		1 & C^*(z)\Phi_\alpha^*(z)^{-1}\e^{-2\theta( z; x,t)}\\
		D^*(z)\Phi_\alpha^*(z)\e^{2\theta( z; x,t)} & 1+C^*(z)D^*(z)
	\end{pmatrix}, & z\in \mbox{Int}(\mathcal{C}_\alpha^-),\ \alpha=1,2,\\[2.2ex]
	P_\infty(k), & \text{elsewhere}.
\end{cases}
\ee
\begin{proposition}
The matrix $W_\infty(z;x,t)$ defined in \eqref{e:defWinf} has the following properties.

\begin{enumerate}
	\item $W_\infty(z;x,t)$ has no jumps across the contours $\mathcal{C}_\alpha^\pm$, for $\alpha\in\{1,2\}$, and therefore is analytic off the segments $\L_\alpha^\pm$.
	
    \item On $\L_\alpha^+$, $W_\infty(z;x,t)$ has the following jump relation
\be\label{e:jumpWinf}
W_{\infty,+}(z;x,t)=W_{\infty,-}(z;x,t)\begin{pmatrix}
	1+2C(z)D(z) & -2D(z)\Phi_\alpha(z)^{-1}\e^{-2\theta(z; x,t)}\\
	-2C(z)(1+C(z)D(z))\Phi_\alpha(z)\e^{2\theta(z; x,t)} & 1+2C(z)D(z)
\end{pmatrix}.
\ee 
The jumps on $\L_\alpha^-$ follow by Schwarz symmetry.
     
     \item Setting 
     \be
     r_1(z)=-\mathrm{i} C(z)\Phi_\alpha(z),\qquad
     r_2(z)=\frac{\mathrm{i} D(z)\Phi_\alpha(z)^{-1}}{1+C(z)D(z)},
     \qquad z\in \L_\alpha^+, 
     \ee
     the jump \eqref{e:jumpWinf} coincides with the jump of RHP \ref{RHP:Mtilde} on $\Sigma$. Under the endpoint condition \eqref{e:vanishing} from Assumption \ref{assume}, both $r_1$ and $r_2$ extend
     continuously to $z=\kappa_\alpha$, and in fact
     \[
     r_1(\kappa_\alpha)=r_2(\kappa_\alpha)=0.
     \]
\end{enumerate}
Consequently, the corresponding NLS solution is recovered from
\[
u(x,t)=\lim_{z\to\infty} 2\mathrm{i} z\, W_{\infty,12}(z;x,t).
\]

\end{proposition}


\appendix
\section*{Appendix}
\setcounter{section}1
\setcounter{subsection}0
\setcounter{equation}0
\setcounter{figure}0
\def\thesection{\Alph{section}}
\def\theequation{\Alph{section}.\arabic{equation}}
\def\thetheorem{\Alph{section}.\arabic{theorem}}
\def\thefigure{\Alph{section}.\arabic{figure}}
\addcontentsline{toc}{section}{Appendix}

\subsection{Jump consistency at the intersecting points}
\label{a:jumpconsistensy}
We first proof the jump consistency for RHP~\ref{RHP:M}.
The jump on $\Sigma_2^+$ can be written in the form
\begin{subequations}
	\begin{equation}
		\restr{V(z)}{z\in \Sigma_2^+} = 
		\e^{-\i(x-x_0^r)z \sigma_3}
		\begin{bmatrix}
			\i r(z_+) \e^{\i \varphi_0^r} & 
			\i \e^{\i \varphi_0^r}  \\[4pt]
			\i [1+ r(z) r^*(z)] \e^{-\i \varphi_0^r} &
			-\i r(z_-) \e^{\i \varphi_0^r}
		\end{bmatrix}
		\e^{\i(x-x_0^r)z \sigma_3}
	\end{equation}
	which admits the factorization
	\begin{equation}\label{V+.two.factor}
		\restr{V(z)}{\mathrlap{z\in \Sigma_2^+}} = 
		\e^{-\i(x-x_0^r)z \sigma_3}
		\begin{bmatrix}
			1 & 0 \\ -r(z_-) & 1 
		\end{bmatrix} 
		\begin{bmatrix}
			0 & \i \e ^{\i\varphi_0^r} \\ 
			\i \e ^{-\i\varphi_0^r} & 0
		\end{bmatrix}
		\begin{bmatrix}
			1 & 0 \\ 
			r(z_+) & 1 
		\end{bmatrix} 
		\e^{\i(x-x_0^r)z \sigma_3}  .
	\end{equation}
\end{subequations}
By symmetry on $\Sigma_2^-$ 
\begin{subequations}
	\begin{equation}
		\restr{V(z)}{z\in \Sigma_2^-} =
		\e^{-\i (x- x_0^r)z \sigma_3}
		\begin{bmatrix}
			-\i r^*(z_-) \e^{-\i\varphi_0^r} & 
			\i [1+ r(z) r^*(z)] \e^{\i\varphi_0^r}   \\[4pt]
			\i \e^{-\i\varphi_0^r} &
			\i r^*(z_+) \e^{-\i\varphi_0^r}
		\end{bmatrix}
		\e^{\i (x- x_0^r)z \sigma_3}
	\end{equation}
	which admits the factorization
	\begin{equation}\label{V-.two.factor}
		\restr{V(z)}{\mathrlap{z\in \Sigma_2^-}} = 
		\e^{-\i(x-x_0^r)z \sigma_3}
		\begin{bmatrix}
			1 & r^*(z_-) \\ 
			0 & 1 
		\end{bmatrix} 
		\begin{bmatrix}
			0 & \i \e ^{\i\varphi_0^r} \\ 
			\i \e ^{-\i\varphi_0^r} & 0
		\end{bmatrix}
		\begin{bmatrix}
			1 & -r^*(z_+) \\ 
			0 & 1 
		\end{bmatrix} 
		\e^{\i(x-x_0^r)z \sigma_3} .
	\end{equation}
\end{subequations}

Let $ \xi_1$  and $\xi_2$ be the points on the real axis where $\Sigma_1$ and $\Sigma_2$ intersect the real axis, respectively.  Then for $z \in \R \setminus \{ \xi_1, \xi_2 \}$ the Riemann-Hilbert problem satisfies the jump condition
\begin{subequations}
	\begin{equation}
		\begin{gathered}
			M_{o+}(z) = M_{o-}(z) V(z), \\   
			V(z) = \e^{-\i(x-x_0^r)z\sigma_3} \begin{bmatrix} 1 + r(z) r^*(z) & r^*(z) \\ r(z) & 1 \end{bmatrix} \e^{-\i(x-x_0^r)z\sigma_3},  
		\end{gathered} 
		\qquad z \in \R \setminus \{ \xi_1, \xi_2 \}
	\end{equation}
	and $V(z)$ admits the triangular factorizations
	\begin{gather}
		\label{V.two.factor}
		V(z) = 
		\e^{-\i(x-x_0^r)z\sigma_3}
		\begin{bmatrix} 1 & r^*(z) \\ 0 & 1 \end{bmatrix} 
		\begin{bmatrix} 1 & 0 \\ r(z) & 1 \end{bmatrix} 
		\e^{\i(x-x_0^r)z\sigma_3}. 
	\end{gather}
\end{subequations}

To check the self-consistency of the jumps at the self-intersection points write 
\begin{equation}
	\begin{aligned}
		&V_1 := \lim_{\substack{z \to \xi_2 \\ z \in (-\infty,\, \xi_2)}} V(z)  \qquad 
		&& V_2 := \lim_{\substack{z \to \xi_2 \\ z \in \Sigma_2^+}} V(z) \\
		&V_3 := \lim_{\substack{z \to \xi_2 \\ z \in (\xi_2,\, \infty)}} V(z)  \qquad 
		&& V_4 := \lim_{\substack{z \to \xi_2 \\ z \in \Sigma_2^-}} V(z)
	\end{aligned}
\end{equation}
self-consistency is consistent with the condition that $ V_1 V_2^{-1} V_3^{-1} V_4 = I$. 
Using the factorizations \eqref{V+.two.factor}, \eqref{V-.two.factor}, \eqref{V.two.factor} we can verify that 
\begin{multline}
	V_1 V_2^{-1} V_3^{-1} V_4 = 
	\e^{-\i(x-x_0^r)z\sigma_3}
	\begin{bmatrix} 1 & r^*(\xi_{2+}) \\ 0 & 1 \end{bmatrix} 
	\begin{bmatrix} 1 & 0 \\ r(\xi_{2+}) & 1 \end{bmatrix} 
	\begin{bmatrix}
		1 & 0 \\ 
		-r(\xi_{2+}) & 1 
	\end{bmatrix} 
	\begin{bmatrix}
		0 & -\i \e ^{\i\varphi_0^r} \\ 
		-\i \e ^{-\i\varphi_0^r} & 0
	\end{bmatrix} \times \\
	\begin{bmatrix}
		1 & 0 \\ r(\xi_{2-}) & 1 
	\end{bmatrix} 
	\begin{bmatrix} 1 & 0 \\ -r(\xi_{2-}) & 1 \end{bmatrix} 
	\begin{bmatrix} 1 & -r^*(\xi_{2-}) \\ 0 & 1 \end{bmatrix} 
	\begin{bmatrix}
		1 & r^*(\xi_{2-}) \\ 
		0 & 1 
	\end{bmatrix} 
	\times \\
	\begin{bmatrix}
		0 & \i \e ^{\i\varphi_0^r} \\ 
		\i \e ^{-\i\varphi_0^r} & 0
	\end{bmatrix}
	\begin{bmatrix}
		1 & -r^*(\xi_{2+}) \\ 
		0 & 1 
	\end{bmatrix}
	\e^{\i(x-x_0^r)z\sigma_3}   = I.
\end{multline}
The computation at $\xi_1$ is similar.  

We now verify the same consistency condition for RHP~\ref{RHP:Mtilde}.
Denote by $V_j^{(M)}$, $j=1,\dots,4$, the corresponding jump matrices of $M$, ordered in the same way as the jump matrices $V_j$ for $M_o$.
Let $D_{\rm NW},D_{\rm NE},D_{\rm SE},D_{\rm SW}$ be the boundary values of this
diagonal multiplier in the four sectors around $\xi_2$, labelled northwest,
northeast, southeast and southwest. Then the four jump matrices for $M$ are
related to those for $M_o$ by
\[
V_1^{(M)}
=
D_{\rm SW}^{-1}V_1 D_{\rm NW},
\qquad
V_2^{(M)}
=
D_{\rm NE}^{-1}V_2 D_{\rm NW},
\]
\[
V_3^{(M)}
=
D_{\rm SE}^{-1}V_3 D_{\rm NE},
\qquad
V_4^{(M)}
=
D_{\rm SE}^{-1}V_4 D_{\rm SW}.
\]
Therefore
\[
	V_1^{(M)}(V_2^{(M)})^{-1}(V_3^{(M)})^{-1}V_4^{(M)}=D_{\rm SW}^{-1} 	V_1  V_2^{-1} V_3^{-1} V_4 D_{\rm SW} =I.
\]
Thus the jumps of the transformed matrix $M$ are also consistent at $\xi_2$.
The verification at $\xi_1$ is identical.

\subsection{Proof of Theorem~\ref{thm:inverse}}
\label{a:proof}
\begin{proof}
The proof follows the original ideal of \cite{Zhou}, and is an analogue of \cite{GJZZ}. 
	Let
	\[
	\widetilde\Sigma:=\mathbb R\cup\Sigma_1\cup\Sigma_2 .
	\]
	We first remove the poles in RHP~\ref{RHP:Mtilde}. For each pole $z_j$, $j=1,\dots, N$, let $\mathfrak{C}_j(x,t)$ denote the residue matrices at $z_j$, that is the rightmost matrix factor in \eqref{e:residue}. 
	Then $\mathfrak{C}_j^\sharp := \begin{bsmallmatrix} 0 &-\i \\ \i & 0 \end{bsmallmatrix} \overline{\mathfrak{C}_j} \begin{bsmallmatrix} 0 &-\i \\ \i & 0 \end{bsmallmatrix}$ denotes te corresponding residue matrix at each $\bar z_j$. 
	Choose pairwise disjoint disks
	\[
	\mathcal{D}_j^+= \mathcal{D}(z_j,\varepsilon),\qquad 
	\mathcal{D}_j^-= \mathcal{D}(\bar z_j,\varepsilon), \qquad j = 1, \dots, N
	\]
	such that
	\[
	\overline{\mathcal{D}_j^\pm}\cap\widetilde\Sigma=\varnothing .
	\]
	Orient all circles \(\partial \mathcal{D}_j^\pm\) counterclockwise and define
	\[
	\widehat M(z;x,t)=
	\begin{cases}
		M(z;x,t)\left(I-\dfrac{\mathfrak{C}_j(x,t)}{z-z_j}\right),
		& z\in \mathcal{D}_j^+,\\[1.2em]
		M(z;x,t) \left(I-\dfrac{\mathfrak{C}_j^\sharp  }{z-\bar z_j}\right),
		& z\in \mathcal{D}_j^-,\\[1.2em]
		M(z;x,t),
		& z\notin \displaystyle\bigcup_{j=1}^N( \mathcal{D}_j^+\cup \mathcal{D}_j^-).
	\end{cases}
	\]
	Then \(\widehat M\) satisfies a holomorphic Riemann--Hilbert problem on
	\[
	\widehat\Sigma
	:=
	\widetilde\Sigma\cup
	\bigcup_{j=1}^N(\partial \mathcal{D}_j^+\cup\partial \mathcal{D}_j^-),
	\]
	with jump matrix 
	\[
	\widehat V(z;x,t)=
	\begin{cases}
		V(z;x,t), & z\in\Gamma,\\[0.8em]
		I-\dfrac{\mathfrak{C}_j(x,t)}{z-z_j}, & z\in\partial \mathcal{D}_j^+,\\[1.2em]
		I-\dfrac{\mathfrak{C}_j^\sharp(x,t)}{z-\bar z_j}, & z\in\partial \mathcal{D}_j^- .
	\end{cases}
	\]
	Therefore RHP~\ref{RHP:Mtilde} is equivalent to the above holomorphic Riemann--Hilbert problem.
	
	Set
	\[
	\widehat w:=\widehat V-I .
	\]
	By \eqref{e:V0}, \eqref{e:residue}, and the assumptions \ref{ass1} and \ref{ass2}, we have
	\[
	\widehat w\in L^2(\widehat\Gamma)\cap L^\infty(\widehat\Gamma).
	\]
	The holomorphic Riemann--Hilbert problem is therefore equivalent to the Beals--Coifman equation
	\[
	(I-C_{\widehat w})\mu=I,
	\qquad
	C_{\widehat w}f:=C_-\big(f\widehat w\big),
	\]
	where $C_-$ denotes the Cauchy projection operator along the contour $\widehat \Gamma$,
	and
	\[
	\widehat M(z;x,t)
	=
	I+\frac{1}{2\pi i}
	\int_{\widehat\Gamma}
	\frac{\mu(s;x,t)\widehat w(s;x,t)}{s-z}\,ds .
	\]
	The same Beals--Coifman and vanishing lemma argument used in \cite{GJZZ} proves that
	\(I-C_{\widehat w}\) is invertible on \(L^2(\widehat\Gamma)\). Hence the holomorphic problem for
	\(\widehat M\) is uniquely solvable. By the equivalence above, RHP~\ref{RHP:Mtilde} is uniquely
	solvable. Expanding the solution for large $z$, we define 
	\[
	u(x,t)=2i\lim_{z\to\infty}zM_{12}(z;x,t).
	\]
	
	Finally, the same differentiability argument for the Beals--Coifman equation as in \cite{GJZZ}
	gives
	\[
	u(x,t)\in C^2(\mathbb R)\times C(\mathbb R^+).
	\]
	This completes the proof.
\end{proof}

\subsection{Overlapping spectral bands without real intersection}\label{s:disjoint}
\begin{figure}[t]
	\centering
	\includegraphics[width=0.2
	\linewidth]{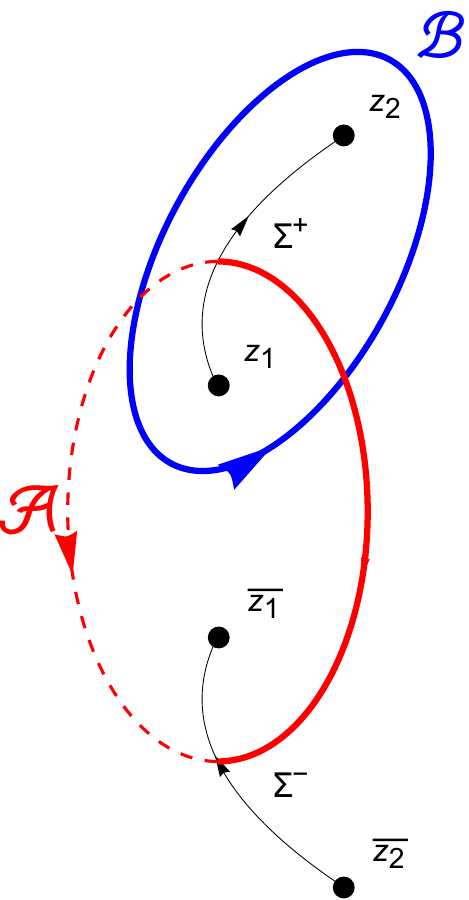}
	\caption{The homology basis for the Riemann surface $\mathcal{X}$ associated with $R^2=(z-z_1)(z-\overline{z_1})(z-z_2)(z-\overline{z_2})$, when the cuts $\Sigma = \Sigma^+ \cup \Sigma^-$ do not intersect the real axis. Solid/dashed lines indicate paths on the first/second sheets of the two sheeted model of $\mathcal{X}$ glued along $\Sigma$. }
	\label{surf1}
\end{figure}

Here we adapt the discussion in Section \ref{sec_elliptic} and Section \ref{sec_perturbative}, to considers steplike elliptic initial data \eqref{initial_data} with completely overlapping spectrum, i.e., $\Sigma^\ell = \Sigma^r  = \Sigma$, but now in the case where $\Sigma \cap \R = \emptyset$. See Figure~\ref{fig:cuts1}. That is, we choose branch points $z_1$ and $z_2$ defining $\Sigma$ for which \eqref{real_spectrum} is not satisfied. 

The first step is to introduce a new Riemann--Hilbert problem encoding the ellipitic waves with spectral bands $\Sigma$ not intersecting the real axis. Compare this to RHP~\ref{RHP:O}.
We lightly abuse notation and use the same symbol $O(z;x,t)$ for its solution.
\begin{RHP}\label{O:RHP}
	Find a $2\times2$ matrix-valued function $O(z) = O(z;x,t)$ which satisfies the following conditions:
	\begin{enumerate}
		\item $O(z;x,t)$ is analytic in $\mathbb{C}\setminus\left\{\Sigma_0\cup\Sigma^-\cup \Sigma^+\right\}$,  where $\Sigma_0=[\overline{z_1},z_1]$  (see Figure~\ref{surf1}). 
		\item $O(z;x,t)$ satisfies the jump conditions $O_+(z;x,t)=O_-(z;x,t)V^{(O)}(z;x,t)$, where
		\be\label{Ojump}
		V^{(O)}(z;x,t)=
		\begin{cases}
			\begin{pmatrix}
				0 & \mathrm{i}\,\e^{2\mathrm{i}\,(xE+tN+\varphi_0)}\\
				\mathrm{i}\,\e^{-2\mathrm{i}\,(xE+tN+\varphi_0)} & 0
			\end{pmatrix}\,, & z\in \Sigma^+\cup \Sigma^-\,,\\
			\e^{\mathrm{i}\Omega\sigma_3}
			\,, & z\in\Sigma_0\,.
		\end{cases}
		\ee
		\item $O(z;x,t)$ satisfies the same normalization at infinity, quartic root singularities at $z\in\{z_1, \overline{z_1},z_2,\overline{z_2}\}$, and Schwarz symmetry as defined in RHP \ref{RHP:O}.
	\end{enumerate}
\end{RHP}

The solution $O(z;x,t)$ is constructed similarly to Proposition~\ref{prop:Osol} with a few modifications to account for the addition of the gap jump along $\Sigma_0$. One now models the Riemann surface $\mathcal{X}$ (cf. \eqref{RS_X} as a two sheeted cover of $\C$ cut and glued along $\Sigma_+ \cup \Sigma_-$. We introduce a homology basis on $\mathcal{X}$ as in Figure~\ref{surf1}. The quasi-momentum and quasi-energy differentials $\d p$ and $\d q$, and the associated phase $\Omega(x,t)$ are still given by \eqref{e:quasidiffs} and \eqref{Omega}; the normalized holomorphic differential \eqref{e:h.diff} is unchanged. On $\C$ we now take the Abel map to be given by \eqref{Abel} where the path of integration does not intersection $\Sigma^+ \cup \Sigma^- \cup \Sigma_0$. One easily verifies that 
\begin{align}
	&\begin{cases}
	    A(z_+) + A(z_-) = 0, &  z \in \Sigma^+,  \\
	A(z_+) + A(z_-) = -1, & z \in \Sigma^-,
	\end{cases}
	&&A(z_+) - A(z_-) = \tau,  \quad z \in \Sigma_0.    
\end{align}
Define 
\begin{equation}
	\tilde\gamma(z) = \left( \frac{ z - z_2}{z-z_1}  \right)^{1/4} \left( \frac{ z - \overline{z_1}}{z-\overline{z_2}}  \right)^{1/4}
\end{equation}
to be analytic in $\C \setminus (\Sigma^+ \cup \Sigma^-)$ and normalized such such $\tilde \gamma(z) \to 1$ as $z \to \infty$ (compare to \eqref{e:gamma}). 
Then one verifies that the solution of RHP~\ref{O:RHP} is given by
\begin{align}\label{defO.no.cross}
	O(z;x,t) = 
	\e^{\i((x-x_0)E + t N+\varphi_0)\sigma_3}
	\begin{bmatrix}
		\frac{\tilde\gamma+\tilde\gamma^{-1}}{2} \tilde H_{11}(z) & \frac{\tilde\gamma-\tilde\gamma^{-1}}{2} \tilde H_{12}(z) \\
		\frac{\tilde\gamma-\tilde\gamma^{-1}}{2} \tilde H_{21}(z) & \frac{\tilde\gamma+\tilde\gamma^{-1}}{2} \tilde H_{22}(z)
	\end{bmatrix}
	\e^{-\i((x-x_0)E + t N+\varphi_0)\sigma_3}\, ,
\end{align}
where
\begin{gather}
	\tilde H_{11}(z) = \frac{\theta_3(0)\theta_3(A(z)-A(\infty)-\frac{\Omega}{2\pi})}{\theta_3(\frac{\Omega}{2\pi})\theta_3(A(z)-A(\infty))}, \quad
	\tilde H_{12}(z) = \frac{\theta_3(0) \theta_3(A(z)+A(\infty)+\frac{\Omega}{2\pi})}{\theta_3(\frac{\Omega}{2\pi}) \theta_3(A(z)+A(\infty))}
	\\
	\tilde H_{21}(z) = \frac{\theta_3(0)\theta_3(A(z)+A(\infty)-\frac{\Omega}{2\pi})}{\theta_3(\frac{\Omega}{2\pi})\theta_3(A(z)+A(\infty))}
	, \quad
	\tilde H_{22}(z) = \frac{\theta_3(0)\theta_3(A(z)-A(\infty)+\frac{\Omega}{2\pi})}{\theta_3(\frac{\Omega}{2\pi})\theta_3(A(z)-A(\infty))}.
\end{gather}
which is identical to \eqref{e:H.theta} up to a halt-period shift in the characteristic of the theta functions: $\theta_4(z;\tau) = \theta_3(z +\tfrac{1}{2}; \tau)$. 
The simultaneous solution of the ZS Lax-pair $W_e(x,t;z)$ corresponding to the elliptic wave $u_e(x,t; x_0 ,\varphi_0)$ is still defined as in \eqref{defWbg} up to our modification of the function $O$. 
One still has 
	\begin{align}\label{jumpW0}
		\begin{split}
			W_e(z_+)=W_e(z_-)
			\begin{cases}
				\i \e^{2\i (x_0 E+\varphi_0)\sigma_3}\sigma_1, & z \in \Sigma, \vspace{8pt}\\
				I, & z \in\Sigma_0,
			\end{cases}
		\end{split}
	\end{align}

The construction of the Jost solution $W^{s}(x,t;z)$ for the ZS spectral problem,  their asymptotic behavior as $x\to\infty$, and the definition of the modified matrix functions $m^s(x,t;z)$  are strictly parallel to those given in \eqref{e:Wpmasy} and \eqref{e:defm}, respectively. 

Substituting the definition of $W^s(x,t;z)$ into the ZS equation \eqref{e:zs}, we find that $m^s(x;z)$ satisfies the exact same ODE as in Lemma \ref{e:ODEm}. 
Consequently, it is governed by the identical Volterra integral equation derived in \eqref{e:Intm} of the Section~\ref{sec_perturbative}. 
The function $W^s(x,t;z)$ has the folliwowing properties

\begin{figure}[t]
	\centering
	\begin{overpic}[width=0.45\linewidth]{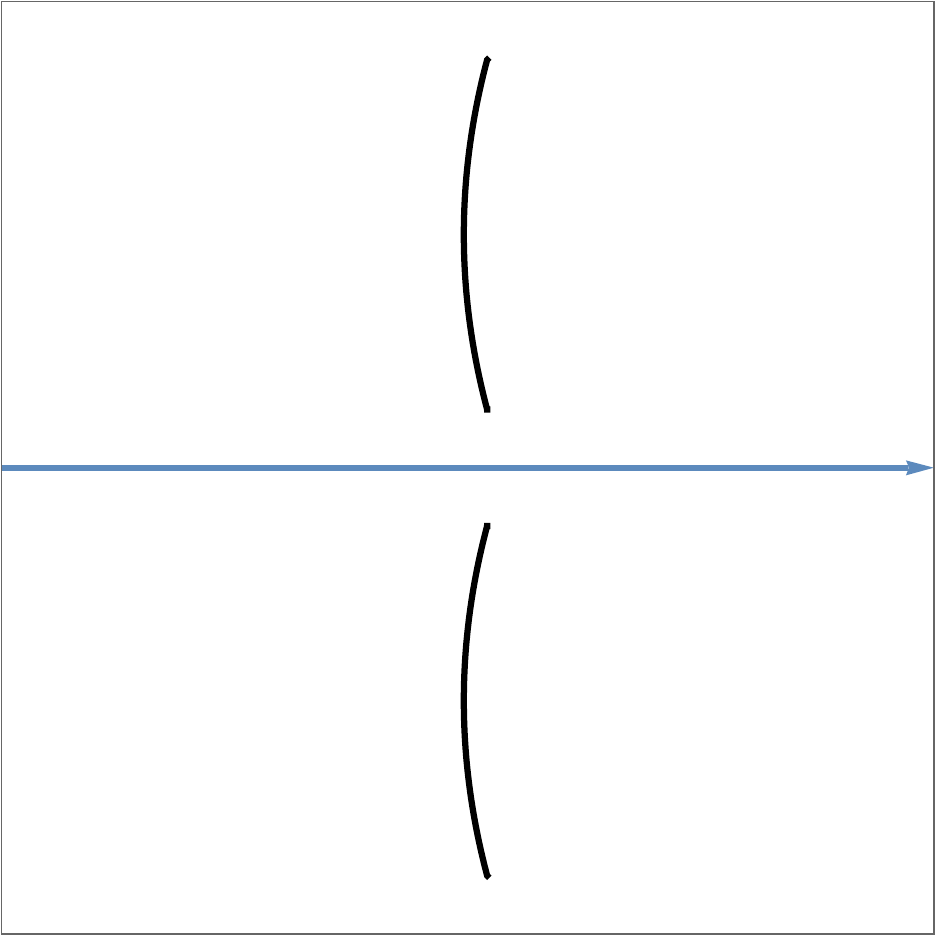}
		\put(60,93){\makebox(0,0)[r]{$z_2$}}
		\put(59,43){\makebox(0,0)[r]{$\overline z_1$}}
		
		\put(54,56){\makebox(0,0)[l]{$z_1$}}
		\put(54,7){\makebox(0,0)[l]{$\overline z_2$}}
		
		\put(51.2,93){\color{red}$\bullet$} 
		\put(51.2,43){\color{red}$\bullet$}  
		\put(51.2,54){\color{red}$\bullet$}  
		\put(51.0,5){\color{red}$\bullet$}  
		\put(94,53){\color{black}$\mathbb{R}$}
		
		\put(40,73){$\Sigma^+$}
		\put(40,23){$\Sigma^-$}
	\end{overpic}
	\caption{The spectral bands $\Sigma=\Sigma^+\cup\Sigma^-$.}
	\label{fig:cuts1}
\end{figure}
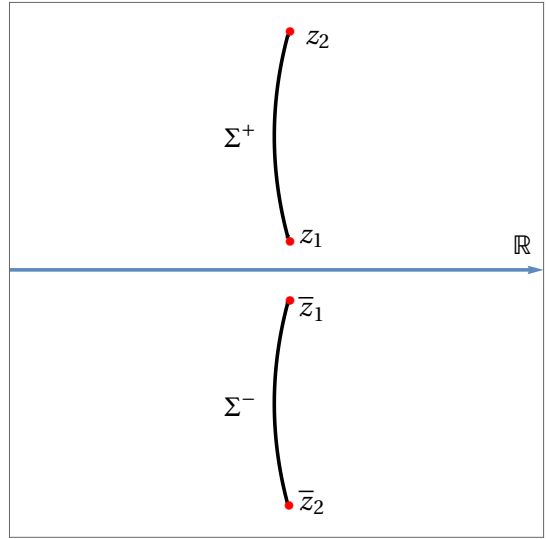


\begin{proposition}\label{Pm}
	Suppose $u_0 - u_e^\ell \in L^1(\R^-)$ and $u_0 - u_e^r \in L^2(\R^+)$. Then $W^s(x;z)$ have the following properties:
	\begin{enumerate}
		\item $W_1^{\ell}(x;z)$ and $W_2^r (x;z)$  are analytic for $z\in\mathbb{C}^+\setminus\Sigma^{+}$;  $W_1^r(x;z)$ and $W_2^{\ell}(x;z)$ are analytic for $z\in\mathbb{C}^-\setminus\Sigma^{-}$.
		
		\item For $s\in\{\ell,r\}$ and $z \in \Sigma^{\pm}$, the boundary values $W^s(x;z_\pm)$ satisfy
	\begin{equation}\label{jumpWs}
		W^s(x;z_+)= W^s(x;z_-)\, \e^{\mathrm{i} \varphi_0^s\sigma_3} (i \sigma_1),  \qquad z \in \Sigma^{\pm}.
	\end{equation}
	
	\item $W^s(x;z)$ has at most fourth-root singularities at $z=z_j$$(j=1,2)$.
	\end{enumerate}
\end{proposition}
The scattering matrix $S(z)$ relating the fundamental solutions $W^{\ell}(z)$ and $W^r(z)$ of \eqref{e:zs} is defined exactly as in \eqref{e:S}.

Using the same piecewise construction \eqref{WB}, we obtain $\hat{M}(z;x)$. This matrix satisfies a Riemann--Hilbert problem identical to RHP \ref{RHP:M} for $M_o(z;x)$, sharing the exact same normalization, symmetries, and singularities, except it is now analytic in $\mathbb{C}\setminus (\mathbb{R}\cup\Sigma^+\cup\Sigma^-)$, and the jump matrices from $\Sigma_1^\pm\cup\Sigma_2^\pm$ are applied to the consolidated contours $\Sigma^\pm$. 

For the subsequent analysis, the auxiliary function $h(z)$, the factorization functions $a_{1}(z)$ and $a_{2}(z)$, and the reflection coefficients $r_1(z)$,  $r_2(z)$ and $\rho(z)$ are defined exactly as in \eqref{e:defh}, \eqref{e:defa12}, \eqref{e:defr1}, \eqref{e:defr2} and \eqref{e:rho}, preserving all properties established in Propositions \ref{p:hep} and \ref{propa12}.

To set up a more general Riemann–Hilbert problem, we symmetrize the factorization by defining
\be
\tilde{M}(z;x)=\hat{M}(z;x)
\begin{cases}
	a_2(z)^{\sigma_3}, & z\in \Complex^+,\\
	a_2^*(z)^{-\sigma_3}, & z\in \Complex^-.
\end{cases}
\ee
Then, $\tilde{M}(z;x)$ satisfies the following Riemann--Hilbert problem:
\begin{RHP}
	Find a $2\times2$ matrix-valued function $\tilde{M}(z;x)$ which satisfies the following conditions:
	\begin{enumerate}
		\item $\tilde{M}(z;x)$ is analytic for $z \in \Complex\setminus (\Real\cup\Sigma^+\cup\Sigma^-)$.
		\item $\tilde{M}(z;x)=I+\mathcal{O}(z^{-1})$, as $z\to\infty$\,.
		\item $\tilde{M}(z;x)$ satisfies the jump condition $\tilde{M}(z_+;x)=\tilde{M}(z_-;x) \tilde{V}(z;x,t=0)$, where $\theta(z;x,t)$ is defined in \ref{e:theta} and 
		 $\tilde{V}(z;x, t=0)$ is given by 
		
		\be
		\tilde{V}(z;x,t=0)=\begin{cases}
			\begin{pmatrix}
				\vspace{6pt}
				\dfrac{1-r_1(z)r_2(z)}{1+r_1(z)r_2(z)} & \dfrac{2\mathrm{i} r_2(z)}{1+r_1(z)r_2(z)}\e^{-2\theta(z;x,0)} \vspace{6pt}
				\\\vspace{8pt}
				\dfrac{2\mathrm{i} r_1(z)}{1+r_1(z)r_2(z)}\e^{2\theta(z;x,0)} & 	\dfrac{1-r_1(z)r_2(z)}{1+r_1(z)r_2(z)} 
			\end{pmatrix}, & z\in\Sigma^+,\\
			\vspace{8pt}
			\begin{pmatrix}
				1+|\rho(z)|^2 & \rho^*(z)\e^{-2\theta(z;x,0)} \vspace{0.5ex}\\
				\rho(z)\e^{2\theta(z;x,0)} & 1
			\end{pmatrix}, & z\in \Real, 
			\vspace{6pt}
			\\
			\begin{pmatrix}
				\dfrac{1-r_1^*(z)r_2^*(z)}{1+r_1^*(z)r_2^*(z)} & 
				\dfrac{2\mathrm{i} r_1^*(z)}{1+r_1^*(z)r_2^*(z)}\e^{-2\theta(z;x,0)}\vspace{0.5ex}
				\\
				\dfrac{2\mathrm{i} r_2^*(z)}{1+r_1^*(z)r_2^*(z)}\e^{2\theta(z;x,0)} &
				\dfrac{1-r_1^*(z)r_2^*(z)}{1+r_1^*(z)r_2^*(z)} 
			\end{pmatrix}, & z\in \Sigma^-.
		\end{cases}
		\ee
	\end{enumerate}
\end{RHP}

\subsection{Proof of Proposition~\ref{prop:gas.error}} \label{app:solgas.error}
\begin{proof}

	Using the analyticity and normalization properties of $P(z)$ and $P_{\infty}(z)$, the error term $\E(z)$ introduced in \eqref{e:error} is analytic in $\C\setminus(\mathcal{C}_1\cup \mathcal{C}_2)$ and satisfies
	\[
	\E(z)=I+O(z^{-1}),\qquad z\to\infty .
	\]
	Its jump condition is
	\[
	\E_+(z)=\E_-(z)V_\E(z),
	\]
	where
	\[
	V_\E(z)
	=
	P_{\infty,-}(z)V_P(z)V_{P_\infty}(z)^{-1}P_{\infty,-}(z)^{-1}.
	\]
	
We first consider $z\in \mathcal{C}_\alpha^+$, $\alpha=1,2$. For brevity, in this proof we write
	\[
	B=B(Nw_\alpha(z)),\qquad \Phi_\alpha=\Phi_\alpha(z).
	\]
	A direct calculation gives
	\[
	V_PV_{P_\infty}^{-1}
	=
	  \begin{pmatrix}
		1+C D-C D\Phi_\alpha B^{-1}
		&
		D(B^{-1}-\Phi_\alpha^{-1})\e^{-2\theta}
		\\[2mm]
		C(1+C D)(B-\Phi_\alpha)\e^{2\theta}
		&
		1+C D-C D\Phi_\alpha^{-1}B
	  \end{pmatrix}.
	\]
	
By Proposition~\ref{prop:asympB} and Assumption~\ref{assume},
	\[
	C(1+C D)(B-\Phi_\alpha)=\mathcal{O}(N^{-1}).
	\]
Also,
	\[
	C D(\Phi_\alpha B^{-1}-1)=\mathcal{O}(N^{-1}),\qquad
	C D(\Phi_\alpha^{-1}B-1)=\mathcal{O}(N^{-1}).
	\]
It remains to estimate the $12$-entry. We write
	\[
	D(z)\left(B^{-1}-\Phi_\alpha^{-1}\right)
	=
	D(z)\Phi_\alpha^{-1}
	\left(\Phi_\alpha B^{-1}-1\right).
	\]
	On compact subsets away from $\kappa_\alpha$, Proposition~\ref{prop:asympB}(i) gives
	\[
	\Phi_\alpha B^{-1}-1=\mathcal{O}(N^{-1}),
	\]
	and hence this term is $\mathcal{O}(N^{-1})$.

	Near $\kappa_\alpha$, we use the local estimate from the proof of Proposition~\ref{prop:asympB}. It gives
	\[
	\Phi_\alpha B^{-1}-1
	=
	\mathcal{O}\left(\min\left\{1,\frac{1}{N|w_\alpha(z)|}\right\}\right).
	\]
	Since $D(\kappa_\alpha)=0$ and
	\[
	\Phi_\alpha=\mathcal{O}\left((z-\kappa_\alpha)^{1/2}\right),
	\]
	we have
	\[
	D(z)\Phi_\alpha^{-1}
	=
	\mathcal{O}\left((z-\kappa_\alpha)^{1/2}\right).
	\]
Therefore,
	\[
	D(z)\left(B^{-1}-\Phi_\alpha^{-1}\right)
	=
	\mathcal{O}\left(
	|z-\kappa_\alpha|^{1/2}
	\min\left\{1,\frac{1}{N|z-\kappa_\alpha|}\right\}
	\right)
	=
	\mathcal{O}(N^{-1/2}).
	\]
Hence, on $\mathcal{C}_\alpha^+$,
	\[
	V_P(z)V_{P_\infty}(z)^{-1}=I+\mathcal{O}(N^{-1/2}).
	\]
	The same estimate holds on $\mathcal{C}_\alpha^-$ by Schwarz symmetry. Since $P_\infty$ has bounded boundary values on $\mathcal{C}_1\cup \mathcal{C}_2$, we obtain
	\[
	V_\E(z)=I+\mathcal{O}(N^{-1/2})
	\]
	uniformly on $\mathcal{C}_1\cup \mathcal{C}_2$.

We can now rewrite the multiplicative jump condition $
	\E_+(z)-\E_-(z)=\E_-(z)(V_\E-I)$, $ z\in \mathcal{C}_1\cup\mathcal{C}_2.$
Using the Plemelj formulas and the normalization at infinity, the RHP for $\E$  is equivalent to the following Cauchy integral equation:
\begin{align}
\E(z)=I+\frac{1}{2\pi \i}\int_{\mathcal{C}_1\cup\mathcal{C}_2}\frac{\E_-(s)(V_\E-I)}{s-z}\d s,\quad z\in\C\setminus(\mathcal{C}_1\cup\mathcal{C}_2).
\end{align}
We then obtain an integral equation for the boundary value $\E_-(z)$
\begin{align}
	\E_-(z)=I+C_-(\E_-(V_{\E}-I))(z),
\end{align}
where $C_-$ is the standard Cauchy boundary operator on $\mathcal{C}_{\alpha}$. We define the operator $C_{V_{\E}-I}(f)=C_-(f(V_{\E}-I))$. Since $C_-$ is bounded operator  on $L^{2}(\mathcal{C}_{\alpha})$, we obtain
\begin{align*}
	\|C_{V_{\E}-I}\|_{L^2\to L^2}\leq \|C_-\|_{L^2}\|V_{\E}-I\|_{L^{\infty}}=\bigo{N^{-1}}.
\end{align*}
This implies immediately $\|C_{V_{\E}-I}\|_{L^2}<1$. By the Neumann series theorem, the operator $(I-C_{V_{\E}-I})$ is invertible in $L^2(\mathcal{C}_\alpha)$. We consider
\begin{align*}
	\E_--I=(I-C_{V_{\E}-I})^{-1}C_{V_{\E}-I}I,
\end{align*}
and correspondingly
\begin{align*}
	\|\E_--I\|_{L^2}\leq \|(I-C_{V_{\E}-I})^{-1}\|_{L^2\to L^2}\|C_{V_{\E}-I}I\|_{L^2}.
\end{align*}
Since $\|C_{V_{\E}-I}I\|_{L^2(\Sigma)} \leq C \|V_{\E} - I\|_{L^2(\Sigma)} = \bigo{N^{-1/2}}$, we deduce that $\|\E_- - I\|_{L^2(\Sigma)} = \bigo{N^{-1/2}}$.

For any $z\in\C\setminus(\mathcal{C}_1\cup\mathcal{C}_2)$ bounded away from the contour, we estimate $\E(z)$ using the Cauchy-Schwarz inequality, we conclude that $\E(z)=I+\bigo{N^{-1/2}}$ uniformly for $z$ in compact of $\C\setminus(\mathcal{C}_1\cup\mathcal{C}_2)$.
\end{proof}


\medskip
\let\em=\it
\makeatletter
\def\@biblabel#1{#1.}
\def\doibase{http://dx.doi.org/}
\def\reftitle#1{``#1''}
\def\booktitle#1{\textit{#1}}
\def\href#1#2{#2}
\makeatother
\small



\begin{thebibliography}{Oi}
	\advance\itemsep-4pt
	
\bibitem{AblowitzSegur81}
	M. J. Ablowitz and H. Segur,
	\booktitle{Solitons and the Inverse Scattering Transform}, SIAM, Philadelphia, 1981
	
\bibitem{BC}
	 R. Beals and R. R. Coifman,  
	  \reftitle{Scattering and inverse scattering for first order systems}, Comm. Pure Appl. Math. {\bf37}, 39--90 (1984)

\bibitem{BGO1}
	M. Bertola, T. Grava and G. Orsatti,   
	\reftitle{Soliton shielding of the focusing nonlinear Schr\"odinger equation},
Phys. Rev. Lett. {\bf130}(12), 127201 (2023)

\bibitem{BGO2}
M. Bertola, T. Grava and G. Orsatti,  
\reftitle{Integrable operators, dbar-problems, KP and NLS hierarchy},  Nonlinearity {\bf37}(8), 085008 (2024)

\bibitem{BGO3}
M. Bertola, T. Grava and G. Orsatti,  
\reftitle{$\overline{\partial}$-problem for the focusing nonlinear Schrödinger equation and soliton shielding}
Bertola, Marco; Grava, Tamara; Orsatti, Giuseppe
Proc. A {\bf481}(2310), 20240764 (2025)

\bibitem{BG14}
G. Biondini and G. Kovačič,
 \reftitle{Inverse scattering transform for the focusing nonlinear Schrödinger equation with nonzero boundary conditions},
J. Math. Phys. {\bf55},  031506 (2014)

\bibitem {BM17} 
G. Biondini and D. Mantzavinos, 
\reftitle{Long-time asymptotics for the focusing nonlinear Schrödinger equation with nonzero boundary conditions at infinity and asymptotic stage of modulation instability}, 
Commun. Pure Appl. Math. {\bf{70}}, 2300--2365 (2017)

\bibitem{BJM18}
M. Borghese, R. Jenkins and K. D. T. R. McLaughlin,
\reftitle{Long time asymptotic behavior of the focusing nonlinear Schr\"{o}dinger equation},
Ann. Inst. H. Poincar\'{e} C Anal. Non Linéaire {\bf 35}, 887--920(2018) 



\bibitem{Deift11}
P. Deift and J. Park,
\reftitle{Long-time asymptotics for solutions of the NLS equation with a delta potential and even initial data}, Int. Math. Res. Not. {\textbf{24}}, 5505--5624 (2011) 



\bibitem{DZ03}
P. Deift and X. Zhou,
\reftitle{Long-time asymptotics for solutions of the NLS equation with initial data in a weighted Sobolev space}, Comm. Pure. Appl. Math. {\textbf{56}}, 1029--1077 (2003) 


\bibitem{dlmf}
\reftitle{NIST Digital Library of Mathematical Functions}, Release 1.2.6 of 2026-03-15. 
F.~W.~J. Olver, A.~B. {Olde Daalhuis}, D.~W. Lozier, B.~I. Schneider, R.~F. Boisvert, C.~W. Clark, B.~R. Miller, B.~V. Saunders, H.~S. Cohl, and M.~A. McClain, eds.


\bibitem{DPVV14}
V. Demontis, B. Prinari, C. Van Der Mee, and F. Vitale,
\reftitle{The inverse scattering transform for the focusing nonlinear Schrödinger equation with asymmetric boundary conditions},
J. Math. Phys. {\bf55}, 101505  (2014)



\bibitem{Deconinck}
B. Deconinck and B. L. Segal,
\reftitle{The stability spectrum for elliptic solutions to the focusing NLS equation},
 Phys. D  {\bf 346}, 1--19 (2017) 
 



\bibitem{DZZ} 
S. Dyachenko, D. Zakharov, and V. Zakharov,
\reftitle{Primitive potentials and bounded solutions of the KdV equation}, Phys. D {\textbf{333}}, 148--156 (2016)


\bibitem{Egorova+24}
I. Egorova, M. Piorkowski, and G. Teschl, 
\reftitle{Asymptotics of KdV shock waves via the Riemann--Hilbert approach}, 
Indiana Univ. Math. J. \textbf{73}(2), 645--690 (2024)


\bibitem{FT}
L. D. Faddeev and L. A. Takhtajan,
\booktitle{Hamiltonian methods in the thoery of solitons}, 
Springer, Berlin, 1987.


\bibitem{GH1}
T. Gallay and M. Haragus. 
\reftitle{Orbital stability of periodic waves for the nonlinear Schrödinger
equation},
J. Dyn. Differ. Equ.  {\bf19}, 825--865 (2007)
\

 \bibitem{GGJM}
M. Girotti, T. Grava, R. Jenkins, K. D. T.-R. McLaughlin,  \reftitle{Rigorous asymptotics of a KdV soliton gas},  Comm. Math. Phys. {\bf384}(2), 733--784 (2021)

 \bibitem{GGJMM}
M. Girotti, T. Grava, R. Jenkins, K. D. T.-R. McLaughlin,  A. Minakov, \reftitle{Soliton versus the gas: Fredholm determinants, analysis, and the rapid oscillations behind the kinetic equation,} Comm. Pure Appl. Math. {\bf76}(11),  3233--3299 (2023)

 \bibitem{GJM}
	M. Girotti, R. Jenkins and K. D. T.-R. McLaughlin,  \reftitle{Long time asymptotics of the generalized soliton gas of the mKdV equation}, in preparation
	
\bibitem{GJZZ}
	T. Grava, R. Jenkins, X. Zhang and Z. Zhang, \reftitle{Direct Scattering of the Focusing Nonlinear Schrödinger Equation with Step-like Oscillatory Initial Data}, 	arXiv:2603.02855 [math.AP]
	
\bibitem{GJZZ2}
	T. Grava, R. Jenkins, X. Zhang and Z. Zhang, \reftitle{Long time asymptotic behavior of NLS with step-like oscillatory initial data}, in preparation

\bibitem{ZS72} 
	V. E. Zakharov and A. B. Shabat, 
	\reftitle{Exact theory of two-dimensional self-focusing and one-dimensional self-modulation of waves in nonlinear media}, 
	Sov. Phys. JETP {\textbf{34}}, 63--69 (1972)
	
\bibitem{Belokolos1994}
	E. D. Belokolos, A. I. Bobenko, V. Z. Enol'skii, A. R. Its and V. B. Matveev, 
	\booktitle{Algebro-geometric Approach to Nonlinear Integrable Equations. Springer Series in Nonlinear Dynamics}, Springer, Berlin, 1994
	
\bibitem{Zhou}
 X. Zhou, 
\reftitle{The Riemann-Hilbert problem and inverse scattering}, SIAM J. Math. Anal. {\bf20}(4), 966--986 (1989)
	
\end{thebibliography}
\end{document}